\documentclass{article}
\usepackage[margin=1in]{geometry}

\usepackage{amsfonts}
\usepackage{amsmath}
\usepackage{amsthm}
\usepackage{amssymb}
\usepackage[all,cmtip]{xy}

\DeclareFontEncoding{OT2}{}{}
  \newcommand{\textcyr}[1]{{\fontencoding{OT2}\fontfamily{wncyr}\fontseries{m}\fontshape{n}
     \selectfont #1}}
\newcommand{\Sha}{{\mbox{\textcyr{Sh}}}}

\newtheorem{theorem}{Theorem}[section]
\newtheorem{lemma}[theorem]{Lemma}
\newtheorem{proposition}[theorem]{Proposition}
\newtheorem{corollary}[theorem]{Corollary}
\newtheorem{application}[theorem]{Application}

\newtheorem{predefinition}[theorem]{Definition}
\newenvironment{definition}{\begin{predefinition}\rm}{\end{predefinition}}
\newtheorem{preremark}[theorem]{Remark}
\newenvironment{remark}{\begin{preremark}\rm}{\end{preremark}}
\newtheorem{preexample}[theorem]{Example}
\newenvironment{example}{\begin{preexample}\rm}{\end{preexample}}

\newcommand{\GG}{\mathbb{G}}
\newcommand{\CC}{\mathcal{C}}
\newcommand{\NN}{\mathbb{N}}
\newcommand{\EE}{\mathbb{E}}
\newcommand{\ZZ}{\mathbb{Z}}
\newcommand{\FF}{\mathbb{F}}
\newcommand{\DD}{\mathbb{D}}
\newcommand{\BB}{\mathbb{B}}
\newcommand{\OO}{\mathcal{O}}
\newcommand{\PP}{\mathcal{P}}
\newcommand{\cU}{{\mathcal U}}
\newcommand{\cF}{{\mathcal F}}
\newcommand{\Hdr}{H^1_{\rm dR}(X_q)}
\newcommand{\Hone}{H^1(X_q, {\mathcal O})}
\newcommand{\Hw}{H^0(X_q, \Omega^1)}
\newcommand{\BdR}{{B_\text{\rm dR}^1}}
\newcommand{\ZdR}{{Z_\text{\rm dR}^1}}
\newcommand{\dx}{{\,dx}}
\newcommand{\dy}{{\,dy}}
\newcommand{\Ttwo}{\langle \times 2 \rangle}

\DeclareMathOperator{\Hom}{Hom}
\DeclareMathOperator{\Span}{Span}
\DeclareMathOperator{\car}{\mathcal{C}}
\DeclareMathOperator{\dime}{dim}
\DeclareMathOperator{\Ker}{Ker}

\title{The Ekedahl-Oort type of Jacobians of Hermitian curves}
\date{}
\author{Rachel Pries \and Colin Weir}

\begin{document}

\maketitle

\begin{abstract}
The Ekedahl-Oort type is a combinatorial invariant of a principally polarized abelian variety $A$
defined over an algebraically closed field of characteristic $p > 0$.
It characterizes the $p$-torsion group scheme of $A$ up to isomorphism. 
Equivalently, it characterizes (the mod $p$ reduction of) the Dieudonn\'e module of $A$ or the 
de Rham cohomology of $A$ as modules under the Frobenius and Vershiebung operators.

There are very few results about which Ekedahl-Oort types occur for Jacobians of curves.
In this paper, we consider the class of Hermitian curves, indexed by a prime power $q=p^n$, 
which are supersingular curves well-known for their exceptional arithmetic properties.
We determine the Ekedahl-Oort types of the Jacobians of all Hermitian curves.
An interesting feature is that their indecomposable factors are determined by 
the orbits of the multiplication-by-two map on $\ZZ/(2^n+1)$, and thus do not depend on $p$.
This yields applications about the decomposition of the Jacobians of Hermitian curves up to isomorphism.

\end{abstract}

\section{Introduction} \label{Introduction}

A crucial fact about a principally polarized abelian variety $A$ 
defined over an algebraically closed field $k$ of characteristic $p >0$
is that the multiplication-by-$p$ morphism of $A$ is inseparable.  
If $A$ has dimension $g$, then $[p]$ factors as $V \circ F$ 
where the Frobenius morphism $F$ is purely inseparable of degree $p^g$
and where $V$ is the Verschiebung morphism.
The isomorphism class of the $p$-torsion group scheme $A[p]$ is determined by the interaction between $F$ and $V$.
It can be characterized by its Ekedahl-Oort type or by the structure of its Dieudonn\'e module.
There are many deep results about the stratification of the moduli space ${\mathcal A}_g$ of principally polarized abelian varieties by Ekedahl-Oort type, 
see especially \cite{O:strat} and \cite{EVdG}.

In contrast, there are almost no results about which Ekedahl-Oort types occur for Jacobians of curves.  
There are existence results for Ekedahl-Oort types of low codimension, for which the 
Jacobians are close to being ordinary \cite{Pr:large}.
There is a complete classification for hyperelliptic curves when $p=2$ \cite{EP}.

In this paper, we determine the Ekedahl-Oort type of the Hermitian curve $X_q$ for every prime power $q$, see Theorem \ref{Torbit}.
More precisely, we determine the structure and multiplicity of each indecomposable factor 
of the Dieudonn\'e module for the $p$-torsion group scheme of the Jacobian of $X_q$.
For the proof, we compute the module structure of $\Hdr$ under $F$ and $V$.
The Hermitian curves are remarkable for their properties over finite fields,
but the Ekedahl-Oort type and the Dieudonn\'e module are geometric invariants.  
Thus we work over $k=\overline{\FF}_p$ throughout the paper.
 
This introduction contains: (1.1) a review of the arithmetic properties of Hermitian curves;
(1.2) a result of Ekedahl that is the starting point for this work;
(1.3) a description of the main result, Theorem \ref{Torbit};
(1.4) an overview of some applications of this result to questions about the isomorphism class of Jacobians of Hermitian curves, about Selmer groups, and 
about the supersingular locus of ${\mathcal A}_g$;
(1.5) a comparison with earlier work; and 
(1.6) an outline of the rest of the paper.

\subsection{Hermitian curves}

The Hermitian curves have received much scrutiny for their remarkable arithmetic properties and applications to combinatorics and coding theory.
For a prime power $q=p^n$, the {\it Hermitian curve} $X_q$ is the curve in $\PP^2$ defined over $\FF_p$ by the homogenization of the equation
\[
X_q: y^q+y=x^{q+1}.
\]
The curve $X_q$ is smooth and irreducible with genus $g=q(q-1)/2$ and it has exactly one point $P_\infty$ at infinity.
The number of points on the Hermitian curve over $\FF_{q^2}$ is $\#X_q\left(\FF_{q^2}\right)=q^3+1$ and the curve $X_q$ is maximal over $\FF_{q^2}$ 
\cite[VI 4.4]{sti09}. 
In fact, $X_q$ is the unique curve of genus $g$ which is maximal over $\FF_{q^2}$ \cite{ruckstich}.
This implies that $X_q$ is the Deligne-Lusztig variety of dimension $1$ associated with the group $G={\rm PGU}(3,q)$ \cite[Proposition 3.2]{HansenDL}.
The automorphism group of $X_q$ is $G$, which has order $q^3(q^2-1)(q^3+1)$, see \cite[Equation 2.1]{GSX};
the Hermitian curves are the only exceptions to the bound of $16g^4$ for the order of the automorphism group of a curve in positive characteristic \cite{sti73}.
They can be characterized as certain ray class fields \cite{lauter}.


The zeta function of $X_q$ is 
\[Z(X_q/\FF_q, t)=\frac{(1+qt^2)^g}{(1-t)(1-qt)},\]
\cite[Proposition 3.3]{HansenDL} and the only slope of the Newton polygon of the $L$-polynomial $L(t)=(1+qt^2)^g$ is $1/2$.
This means that $X_q$ is {\it supersingular} for every prime power $q$.  
The supersingular condition is equivalent to the condition that the Jacobian ${\rm Jac}(X_q)$ is isogenous to a product of supersingular elliptic curves \cite[Theorem 4.2]{O:sub}.  
It also implies that ${\rm Jac}(X_q)$ has no $p$-torsion points over $\overline{\FF}_p$.

\subsection{A result of Ekedahl} \label{Sekedahl}

It is well-known that the Jacobian of the Hermitian curve $X_p:y^p+y=x^{p+1}$ is {\it superspecial}, see Section \ref{Sprankanumber} for definitions.
Briefly, the superspecial condition is equivalent to the condition that the 
Jacobian ${\rm Jac}(X_p)$ is isomorphic to a product of supersingular elliptic curves \cite[Theorem 2]{oort75}, see also \cite[Theorem 4.1]{Nygaard}. 
Equivalently, (the mod $p$ reduction of) the Dieudonn\'e module of the $p$-torsion group scheme of ${\rm Jac}(X_p)$ is isomorphic to the sum of $g$ copies 
of the Dieudonn\'e module of a supersingular elliptic curve:
\begin{equation} \label{En=1}
\DD({\rm Jac}(X_p)) \simeq (\EE/\EE(F+V))^{g}.
\end{equation}
(Here $\EE=k[F,V]$ is the non-commutative ring generated by semi-linear operators $F$ and $V$ with the relations 
$FV=VF=0$ and $F \lambda = \lambda^p F$ and $\lambda V=V \lambda^p$ for all $\lambda \in k$
and $\EE(A_1, \ldots)$ denotes the left ideal of $\EE$ generated by $A_1, \ldots$).
The easiest way to prove that ${\rm Jac}(X_p)$ is superspecial is to show that the Cartier operator is the zero operator on $H^0(X_p, \Omega^1)$, 
which implies that the kernel of Frobenius $F$ is the kernel of Verschiebung $V$ on the Dieudonn\'e module. 

There is an upper bound $g \leq p(p-1)/2$ for the genus of a superspecial curve in characteristic $p$, \cite[Theorem 1.1]{Ekedahl} 
and this upper bound is realized by $X_p$.
For $n \geq 2$, it is thus impossible for $X_q$ to be superspecial.
 
\subsection{Main result}  

In this paper, we determine the $\EE$-module structure of the Dieudonn\'e module $\DD(X_q):=\DD({\rm Jac}(X_q)[p])$ for all prime powers $q=p^n$.
This is the same as determining the isomorphism class of the $p$-torsion group scheme of ${\rm Jac}(X_q)$.
In the main result, see Theorem \ref{Torbit}, we prove that the distinct indecomposable factors 
of $\DD(X_q)$ are in bijection with 
orbits of $\ZZ/(2^n+1) -\{0\}$ under $\Ttwo$ where $\Ttwo$ denotes multiplication-by-two.
The structure of each factor is determined by the combinatorics of the orbit, as explained in Section \ref{Sstructure2}.
In particular, the $a$-number of each factor is odd.
We also determine the multiplicities of the factors.
While these multiplicities depend on $p$, 
the structure of each indecomposable factor depends only on $n$.
Theorem \ref{Torbit} determines the Ekedahl-Oort type $\nu$ of ${\rm Jac}(X_{p^n})$, although an explicit formula for $\nu$ is not easy to write down for general $n$.
In particular, $\nu$ has $2^{n-1}$ {\it break points} where the behavior of the 
Ekedahl-Oort sequence switches between the states of being constant and increasing, see Section \ref{Seotype} and Corollary \ref{Ckeyvalue}.

Examples of $\DD(X_{p^n})$ for small $n$ appear in Section \ref{Sintuition} and Example \ref{Ecasen=4}.
When $n=2$, the $\Ttwo$ map on $\ZZ/5-\{0\}$ has one orbit $\{1,2,4,3\}$.
Theorem \ref{Torbit} implies that the Dieudonn\'e module of ${\rm Jac}(X_{p^2})$ decomposes into $g/2$ copies of the 
Dieudonn\'e module of a supersingular (but not superspecial) abelian surface:
\begin{equation} \label{En=2}
\DD(X_{p^2})= (\EE/\EE(F^2+V^2))^{g/2}.
\end{equation}

For one of the applications, we determine that the $\EE$-module $\EE/\EE(F+V)$ appears as a factor of $\DD(X_q)$ if and only if $n$ is odd, in which case it 
appears with multiplicity $\left(p(p-1)/2\right)^n$, see Corollary \ref{Canumber1}.

\subsection{Applications}

Theorem \ref{Torbit} gives partial information about the decomposition of ${\rm Jac}(X_q)$, up to isomorphism, into indecomposable abelian varieties, see Section \ref{SJacDec}.
For example, when $n$ is a power of $2$, we prove that the dimension of each factor in such a decomposition is a multiple of $n$.
For another application, let the {\it elliptic rank} of an abelian variety $A$ be the largest non-negative integer $r$ such that there exist elliptic curves 
$E_1, \ldots, E_r$ and an abelian variety $B$ of dimension $g-r$ and an isomorphism $A \simeq B \times (\times_{i=1}^r E_i)$ of abelian varieties without polarization. 

\begin{application} \label{Appellipticrank}
If $n$ is even, then the elliptic rank of ${\rm Jac}(X_{p^n})$ is $0$.
If $n$ is odd, then the elliptic rank of ${\rm Jac}(X_{p^n})$ is at most $\left(p(p-1)/2\right)^n$.
\end{application}

The second application is about the Selmer groups for the multiplication-by-$p$ isogeny of a constant elliptic curve $E$ over the function field of a Hermitian curve, 
see Section \ref{SSelmer}.
The third application is about Ekedahl-Oort strata with $a$-number just less than $g/2$ which 
intersect but are not contained in the supersingular locus of ${\mathcal A}_g$, see Section \ref{Ssupersingular}.

\subsection{Earlier work}

After finishing this research, we became aware of some other results about the cohomology of Hermitian curves.
In \cite{hj90}, the authors study filtrations of the crystalline cohomology of Hermitian curves with the motivation of understanding filtrations of Weyl modules of algebraic groups. 
In \cite{dum95, dum99}, Dummigan analyzes ${\rm Jac}(X_q)$ viewed as a constant abelian variety over the function field of $X_q$.
His motivation is to study the structure of the Tate-Shafarevich group $\Sha$ of ${\rm Jac}(X_q)$ and the determinant of the lattice ${\rm End}_{\FF_{q^2}}({\rm Jac}(X_q))$.
In particular, he proves that $\Sha$ is trivial if and only if $n \leq 2$ and the smallest power of $p$ annihilating $\Sha$ is $p^{\lfloor n/3\rfloor}$.
He uses the alternative equation $u^{q+1}+v^{q+1}+w^{q+1}=0$ for $X_q$ to find a basis for the crystalline cohomology of the lifting $X_q^*$ of $X_q$ over the Witt vectors
which is convenient for computing the action of $F$. 
As part of \cite{dum95}, Dummigan finds the structure of $\Hdr$ as an $\FF_{q^2}[G]$-module and as an $\FF_p[G]$-module.  

It appears that the blocks defined in Definition \ref{Dblock} are the indecomposable $\FF_{q^2}[G]$-modules of $\Hdr$.
It might be possible to cut Section \ref{SFV} of this paper by referring to \cite{dum95}. 
We decided to include the material in Section \ref{SFV} because the method in \cite{dum95}  
relies heavily on a property of the Hermitian curve which is quite rare, 
namely that there is a decomposition of $\Hdr$ into one-dimensional eigenspaces for a group of prime-to-$p$ automorphisms.
In contrast, the method in Section \ref{SFV} involving the action of $F$ and $V$ on $\Hdr$ can be used to compute the Ekedahl-Oort type for a wide class of Jacobians.  
In addition, our description of the combinatorial structure in terms of orbits of $\Ttwo$ may be easier to work with than 
the {\it circle diagrams} of  \cite[Section 7]{dum95}.

\subsection{Outline of paper}

Section \ref{Snotation} contains background material about $p$-torsion group schemes and the de Rham cohomology and some $p$-adic formulae.  
In Section \ref{Sintuition}, we give examples and explain the case $n=3$ in order to give a conceptual overview of the combinatorial structures found in the paper.
The action of $F$ and $V$ on $\Hdr$ is computed in Section \ref{SFV}.  A decomposition of $\Hdr$ into 
blocks permuted by $F$ and $V$ is developed in Section \ref{SHdr}.
Section \ref{Sthmorbit} contains the main theorem about the bijection between indecomposable factors of 
the Dieudonn\'e module and orbits of $\Ttwo$.  The applications are in Section \ref{Sapplication}.

The first author was partially supported by NSF grant DMS-11-01712. 
The second author was partially supported by NSERC and AITF.
We would like to thank J. Achter, A. Hulpke, and F. Oort for helpful conversations and the referees for their valuable comments.

\section{Notation and background} \label{Snotation}

\subsection{Classification of $p$-torsion group schemes}

\subsubsection{Frobenius and Verschiebung}

Suppose $A$ is a principally polarized abelian variety of dimension $g$ defined over $k$.
For example, $A$ could be the Jacobian of a $k$-curve of genus $g$.
Consider the multiplication-by-$p$ morphism $[p]:A \to A$ which is a finite flat morphism of degree $p^{2g}$.
It factors as $[p]=V \circ F$.  Here $F:A \to A^{(p)}$ is the relative Frobenius morphism 
coming from the $p$-power map on the structure sheaf; it is purely inseparable of degree $p^g$.  
The Verschiebung morphism 
$V:A^{(p)} \to A$ is the dual of $F_{A^{\rm dual}}$.  

\subsubsection{The $p$-torsion group scheme}

The {\it $p$-torsion group scheme} of $A$, denoted $A[p]$, is the kernel of $[p]$.  
It is a finite commutative group scheme annihilated by $p$, again having morphisms $F$ and $V$.
The polarization of $A$ induces a symmetry on $A[p]$ as defined in \cite[5.1]{O:strat}; when $p >2$, this is 
an anti-symmetric isomorphism from $A[p]$ to the Cartier dual group scheme $A[p]^{{\rm dual}}$ of $A[p]$.
By \cite[9.5]{O:strat}, the $p$-torsion group scheme 
$A[p]$ is a polarized ${\rm BT}_1$ group scheme over $k$ 
(short for polarized Barsotti-Tate truncated level 1 group scheme), as defined in \cite[2.1, 9.2]{O:strat}.
The rank of $A[p]$ is $p^{2g}$.

Here is a brief summary of the classification \cite[Theorem 9.4 \&12.3]{O:strat} 
of polarized ${\rm BT}_1$ group schemes over $k$ in terms of Dieudonn\'e modules and Ekedahl-Oort type;
other useful references are \cite{Kraft} (unpublished - without polarization) and \cite{M:group} (for $p \geq 3$).
When $p=2$, there are complications with the polarization which are resolved in \cite[9.2, 9.5, 12.2]{O:strat}.

\subsubsection{Covariant Dieudonn\'e modules}

One can describe the group scheme $A[p]$ using (the modulo $p$ reduction of) the {\it covariant Dieudonn\'e module}, see e.g., \cite[15.3]{O:strat}.
This is the dual of the contravariant theory found in \cite{Demazure}.
Briefly, consider the non-commutative ring $\EE=k[F,V]$ generated by semi-linear operators $F$ and $V$ 
with the relations $FV=VF=0$ and $F \lambda = \lambda^p F$ and $\lambda V=V \lambda^p$ for all $\lambda \in k$.
Let $\EE(A_1, \ldots, A_r)$ denote the left ideal $\sum_{i=1}^r \EE A_i$ of $\EE$ generated by $\{A_i \mid 1 \leq i \leq r\}$.
The category of commutative group schemes over $k$ annihilated by $p$ is equivalent to the category of finite left $\EE$-modules.
Given a ${\rm BT}_1$ group scheme $\GG$ over $k$
we denote by $D(\GG)$ the Dieudonn\'e module of $\GG$.
If $\GG$ has rank $p^{2g}$, then $D(\GG)$ has dimension $2g$ as a $k$-vector space. 
For example, the Dieudonn\'e module of a supersingular elliptic curve is $\EE/\EE(F+V)$, \cite[Ex.\ A.5.4]{G:book}.

\subsubsection{The $p$-rank and $a$-number} \label{Sprankanumber}

Two invariants of (the $p$-torsion of) an abelian variety are the $p$-rank and $a$-number.
The {\it $p$-rank} of $A$ is $f=\dime_{\FF_p} \Hom(\mu_p, A[p])$
where $\mu_p$ is the kernel of Frobenius on $\GG_m$.
Then $p^f$ is the cardinality of $A[p](k)$.
The {\it $a$-number} of $A$ is $a=\dime_k \Hom(\alpha_p, A[p])$ 
where $\alpha_p$ is the kernel of Frobenius on $\GG_a$.
It is well-known that $0 \leq f \leq g$ and $1 \leq a +f \leq g$.
Then $A$ is {\it superspecial} if $a=g$.
The $p$-rank of $\GG=A[p]$ is the dimension of $V^g D({\mathbb G})$.
The $a$-number of $A[p]$ equals $g-{\rm dim}(V^2D({\mathbb G}))$ \cite[5.2.8]{LO}.

\subsubsection{The Ekedahl-Oort type} \label{Seotype}

As in \cite[Sections 5 \& 9]{O:strat}, the isomorphism type of a ${\rm BT}_1$ group scheme $\GG$ over $k$ can be encapsulated into combinatorial data.
If $\GG$ is symmetric with rank $p^{2g}$, then there is a {\it final filtration} $N_1 \subset N_2 \subset \cdots \subset N_{2g}$ 
of $D(\GG)$ as a $k$-vector space which is stable under the action of $V$ and $F^{-1}$ such that $i={\rm dim}(N_i)$, \cite[5.4]{O:strat}.
If $w$ is a word in $V$ and $F^{-1}$, then $wD(\GG)$ is an object in the filtration; in particular, $N_g = V D(\GG) =F^{-1}(0)$.

The {\it Ekedahl-Oort type} of $\GG$, also called the {\it final type},
is $\nu=[\nu_1, \ldots, \nu_g]$ where ${\nu_i}={\rm dim}(V(N_i))$.
The $p$-rank is ${\rm max}\{i \mid \nu_i=i\}$ and the $a$-number equals $g-\nu_g$.
The Ekedahl-Oort type of $\GG$ does not depend on the choice of a final filtration.
There is a restriction $\nu_i \leq \nu_{i+1} \leq \nu_i +1$ on the final type.
There are $2^g$ Ekedahl-Oort types of length $g$ since all sequences satisfying this restriction occur.   
By \cite[9.4, 12.3]{O:strat}, there are bijections between (i) Ekedahl-Oort types of length $g$; (ii) polarized ${\rm BT}_1$ group schemes over $k$ of rank $p^{2g}$;
and (iii) principal quasi-polarized Dieudonn\'e modules of dimension $2g$ over $k$.

In the terminology of \cite[Section 2.2]{EVdG}, an integer $1 \leq i \leq g$ is a {\it break point} of $\nu$ if either 
$\nu_{i-1}=\nu_i \not = \nu_{i+1}$ or $\nu_{i-1} \not =\nu_i = \nu_{i+1}$.  The Ekedahl-Oort type is determined by its break points, since 
these are the indices at which the behavior of the sequence $\nu_i$ switches between the states of being constant and increasing.
The break points are the last indices of the {\it canonical fragments} of $\nu$.	

\subsubsection{The de Rham cohomology} \label{Sdefderham}

By \cite[Section 5]{Oda}, there is an isomorphism of $\EE$-modules between the Dieudonn\'e module of the $p$-torsion group scheme ${\rm Jac}(X_q)[p]$ 
and the de Rham cohomology group $\Hdr$.  

Applying \cite[Section 5]{Oda}, there is the following description of $\Hdr$. 
Recall that $\dim_k \Hdr = 2g$.
Consider the open cover $\cU$ of $X_q$ given by $U_1=X_q \setminus \{P_\infty\}$ 
and $U_2 = X_q \setminus \{(0,y) \mid y^q+y=0\}$.
For a sheaf $\cF$ on $X_q$, let
\begin{eqnarray*}
\CC^0(\cU, \cF) &:=& \{\kappa = (\kappa_1, \kappa_2) \mid \kappa_i\in \Gamma(U_i, \cF)\},\\
\CC^1(\cU, \cF) &:=& \{\phi \in \Gamma(U_1 \cap U_2, \cF)\}.
\end{eqnarray*}
The coboundary operator $\delta: \CC^0(\cU, \cF) \to \CC^1(\cU, \cF)$ is defined by
$\delta \kappa = \kappa_i - \kappa_j$.

The closed de Rham cocycles are defined by
\[
\ZdR(\cU) := \{(\phi, \omega)\in \CC^1(\cU, \OO) \times \CC^0(\cU, \Omega^1) \mid d\phi = \delta\omega\},
\]
that is, $d\phi = \omega_1 - \omega_2$.  
The de Rham coboundaries are defined by
\begin{equation*}
\BdR(\cU) := \{(\delta \kappa, d\kappa)\in \ZdR(\cU) \mid \kappa\in C^0(\cU, \OO)\}.
\end{equation*}
Finally,
\[
\Hdr \cong \Hdr(\cU) := \ZdR(\cU) / \BdR(\cU).
\]


There is an injective homomorphism $\lambda:\Hw \to \Hdr$
denoted informally by $\omega \mapsto (0, \omega)$ where the second coordinate 
is defined by $\omega_i=\omega \vert_{U_i}$.
This map is well-defined since $d(0)= \omega \vert_{U_1} - \omega \vert_{U_2}= \delta \omega$.
It is injective because, if $(0, \omega) \equiv (0, \omega') \bmod{\BdR(\cU)}$,
then $\omega - \omega' = d\kappa$ where $\kappa \in C^0(\cU, \OO)$ is such that 
$\delta \kappa=0$;
thus $\kappa$ is a constant function on $X$ and so $\omega - \omega' = 0$.

There is another homomorphism $\gamma: \Hdr \to \Hone$ sending
the cohomology class of $(\phi,\omega)$ to the cohomology class of $\phi$. 
The choice of cocycle $(\phi, \omega)$ does not matter, since 
the coboundary conditions on $\Hdr$ and $\Hone$ are compatible.
The homomorphisms $\lambda$ and $\gamma$ fit into a short exact sequence
\begin{equation*}
0 \to \Hw \xrightarrow{\lambda} \Hdr
\xrightarrow{\gamma} \Hone \to 0.
\end{equation*}
In Subsection \ref{Sbasis}, 
we construct a suitable section $\sigma: \Hone \to \Hdr$ of $\gamma$ as $k$-vector spaces.

\subsubsection{The action of Frobenius and Verschiebung on $\Hdr$}
\label{sub:fv}

The Frobenius and Verschiebung operators $F$ and $V$ act on $\Hdr$ as follows:
\[
F(f, \omega) := (f^p, 0)
\quad\text{and}\quad
V(f, \omega) := (0, \car(\omega)),
\]
where $\car$ is the Cartier operator \cite{Cartier} on the sheaf $\Omega^1$.
The operator $F$ is $p$-linear and $V$ is $p^{-1}$-linear.
In particular, $\ker(F) = \Hw={\rm im}(V)$.

The three principal properties of the Cartier operator are that
it annihilates exact differentials, preserves logarithmic ones, and is $p^{-1}$-linear.
The Cartier operator can be computed as follows.
The element $x\in k(X_q)$ forms a $p$-basis of $k(X_q)$ over $k(X_q)^p$,
i.e., every $z\in k(X_q)$ can be written as
$z := z_0^p + z_1^p x + \cdots + z_{p-1}^p x^{p-1}$
for uniquely determined $z_0, \ldots, z_{p-1} \in k(X_q)$.
Then $\car(z \dx/x) := z_{0} \dx/x$.

\subsection{Examples and conceptual overview} \label{Sintuition}

We illustrate the structure of the $p$-torsion group schemes of the Jacobians of the Hermitian curves $X_{p^n}$ for $n \leq 3$ 
as a way of motivating later computations.  The case $n=4$ can be found in Example \ref{Ecasen=4}.

The $p$-rank of $X_q$ is zero since $X_q$ is supersingular.
Let $r_{n,i}$ denote the rank of the $i$th iterate of the Cartier operator $\car$ on $\Hw$. 
The $a$-number of $X_q$ is $a_n=g-r_{n,1}$.
In Proposition \ref{PrankC}, we prove that \[r_{n,i}=p^n(p+1)^i(p^{n-i}-1)/2^{i+1}.\]
 
\subsubsection{The case $n=1$} \label{Ecasen=1}

When $n=1$, then the rank of $\car$ is $r_{1,1}=0$ and so the $a$-number is $a_1=g$.  By definition, 
$X_1$ is superspecial.  The Ekedahl-Oort type of ${\rm Jac}(X_{p})[p]$ is $[0, \ldots, 0]$ and $\DD(X_p)= (\EE/\EE(F+V))^{g}$ as in \eqref{En=1}.

\subsubsection{The case $n=2$}

When $n=2$, then $r_{2,1}=g/2$ and $r_{2,2}=0$.  The Ekedahl-Oort type $\nu=[\nu_1, \ldots, \nu_g]$ has values $\nu_{g}=g/2$ and $\nu_{g/2}=0$.  
By the numerical restrictions on $\nu$ found in Section \ref{Seotype}, this implies that $\nu_i=0$ and $\nu_{g/2+i}=i$ for $1 \leq i \leq g/2$, so that $\nu=[0, \ldots, 0, 1, 2, \ldots, g/2]$.

Using \cite[9.1]{O:strat}, the Dieudonn\'e module is generated by variables $Z_i$ for $1 \leq i \leq 2g$ which are defined in terms of variables $Y_i$ and $X_i$ for $1 \leq i \leq g$.
Imprecisely speaking, the variables $Y_i$ are used (in reverse order) for the indices where the value in the Ekedahl-Oort type stays constant, and the variables $X_i$ are used 
for the indices where the value in the Ekedahl-Oort type is increasing.
In the case $n=2$, this yields:

\[\begin{array}{|c||c|c|c|c|}
\hline
i & 1 \leq i \leq g/2 & 1+ g/2 \leq i \leq g & g+1 \leq i \leq 3g/2 & 1+3g/2 \leq i \leq 2g \\
\hline
Z_i & Y_{g+1-i} & X_{i-g/2} & Y_{1-i+3g/2} & X_{i-g}\\
\hline
\end{array}
\] 

For $1 \leq i \leq g$, the actions of Frobenius and Verschiebung are defined by the rules:
\[F(X_i)=Z_i, \ F(Y_i)=0, \ V(Z_i)=0, \ V(Z_{2g+1-i})= \pm Y_i.\]

With respect to the ordered variables $Z_1, \ldots, Z_{2g}$, the action of $F$ and $V$ are given by the following (each entry represents a square matrix of size $g/2$):

\[
F=\begin{pmatrix}
0 & I & 0 & 0 \\
0 & 0 & 0 & I \\
0 & 0 & 0 & 0 \\
0 & 0 & 0 & 0 \\
\end{pmatrix}, \ 
V=\begin{pmatrix}
0 & 0 & I & 0 \\
0 & 0 & 0 & 0 \\
0 & 0 & 0 & -I \\
0 & 0 & 0 & 0 \\
\end{pmatrix}.
\]

Thus $\DD(X_{p^2})$ is generated by $Z_i$ with relation $(F^2+V^2)Z_i=0$ 
for $1+3g/2 \leq i \leq 2g$, proving $\DD(X_{p^2})= (\EE/\EE(F^2+V^2))^{g/2}$ as in \eqref{En=2}.

\subsubsection{The case $n=3$} \label{Sintuition3}

For $n =3$ (or larger), the information gleaned from ranks of iterates of the Cartier operator is not enough
to determine the structure of the $p$-torsion group scheme.
When $n=3$, $\nu_g=r_{3,1}$, $\nu_{r_{3,1}}=r_{3,2}$ and $\nu_{r_{3,2}}=0$.  
Since $r_{3,1}=2r_{3,2}$, the values $\nu_i$ remain $0$ for $1 \leq i \leq r_{3,2}$ 
and then increase by one at each index for $r_{3,2} < i \leq r_{3,1}$.
Among the indices $r_{3,1} < i \leq g$, it is clear that the values $\nu_i$ must rise 
by a combined total of $r_{3,2}$.  In other words, the value $\nu_i$ must 
increase at somewhat more than half of the indices $i$ in this range, but it is not clear at which ones. 
    
More information is required to determine the values $\nu_i$ for $r_{3,1} < i < g$, 
specifically, the full structure of $\Hdr$ as an $\EE$-module.
We compute the actions of $F$ and $V$ on a basis for $\Hdr$ in Section \ref{SHone}.  The results are numerically intricate and it is not initially clear how to find 
a filtration $N_1 \subset N_2 \subset \cdots \subset N_{2g}$ of $\Hdr$
which is stable under the action of $V$ and $F^{-1}$.


At this stage, computer calculations for small $p$ convinced us that the values $\nu_i$ stay as small 
as possible in the range $r_{3,1} < i \leq g$; in other words, 
that $\nu_i=r_{3,2}$ for $r_{3,1} < i \leq g-r_{3,2}$ 
and then $\nu_i$ increases by one at each index in the range $g-r_{3,2} < i \leq g$.
We came to expect that the Ekedahl-Oort type has the break points $r_{3,2}$, $r_{3,1}$, and $g-r_{3,2}$ when $n=3$
and considered the implications of this hypothesis.

This hypothesis implies that the interval $1 \leq i \leq 2g$ is divided into $8$ canonical fragments, 
six of size $r_{3,2}$ and two of size $g-3r_{3,2}$, 
for which the sequence $\nu_i$ switches between the states of being constant and increasing. 
Labeling these as $B_1, \ldots, B_8$, the technique of \cite[9.1]{O:strat} implies that, for $1 \leq i \leq 8$, 
\[F(B_i)=B_{i/2} {\rm \ if \ }i {\rm \ even \ and \ } F(B_i)=0 {\rm \ if \ }i {\rm \ odd};\]
and, for $1 \leq i \leq 4$, 
\[V(B_i)=0 {\rm \ and \ } V(B_{4+i})=\pm B_{2i-1}.\] 
This implies that $\DD(X_{p^3})$ is generated by the $r_{3,2}$ variables in $B_8$ and the $g-3r_{3,2}=(\frac{p(p-1)}{2})^3$ variables in $B_6$, 
subject to the relations that $F^3+V^3=0$ on $B_8$ and $F+V=0$ on $B_6$.
On each block $B_i$, exactly one of $F^{-1}$ and $V$ is defined, 
and the action on the blocks is the same as $\Ttwo$ on $\ZZ/9-\{0\}$.

To prove this, we find a decomposition of $\Hdr$ into blocks $B_i$, which is compatible with the condition that the final filtration must be a refinement of the filtration:
\[0=T_0 \subset T_1 \subset T_2 \subset \cdots \subset T_8,\]
where
\[\begin{array}{|c||c|c|c|c|c|c|c|c|}
\hline
i & 1 & 2 & 3 & 4 & 5 & 6 & 7 & 8 \\
\hline
T_i/T_{i-1} & B_1 & B_5 & B_3 & B_7 & B_2 & B_6 & B_4 & B_8\\
\hline
\end{array}.\]
For example, this shows $\Hw={\rm Span}(B_1, B_3, B_5, B_7)$ 
and $\Hone={\rm Span}(B_2, B_4, B_6, B_8)$.

We assign basis vectors of $\Hw$ and $\Hone$ to blocks based on the following rules, see Sections \ref{Sbvd} and \ref{Scongdec}.
Given $i, j \geq 0$ such that $i+j \leq p^3-2$, 
consider the $p$-adic expansions $i=i_0 + i_1 p + i_2 p^2$ and $j=j_0 + j_1p +j_2 p^2$.
Define $b_0, b_1 \in \ZZ/2$ by $b_0=0$ iff $i_0 + j_0 < p-1$ and $b_1 = 0$ iff $i_0+i_1p+j_0+j_1p < p^2-1$.
To a basis vector $\omega_{i,j} = x^iy^j dx$ of $\Hw$, we assign the vector $(b_0,b_1,1) \in (\ZZ/2)^3$.
To a basis vector $f_{i,j}=\frac{1}{x^iy^j} \frac{y^{q-1}}{x}$ of $\Hone$, 
we assign the vector $(b_0,b_1,0) \in (\ZZ/2)^3$.
We then assign the vectors to blocks by:

\[\Hone \ 
\begin{array}{|c||c|c|c|c|}
\hline
{\rm vector} & (0,0,0) & (0,1,0) & (1,0,0) & (1,1,0) \\
\hline
{\rm block} & B_8 & B_6 & B_4 & B_2\\
\hline
\end{array},\]
and 
\[
\Hw \ 
\begin{array}{|c||c|c|c|c|}
\hline
{\rm vector} & (0,0,1) & (0,1,1) & (1,0,1) & (1,1,1) \\
\hline
{\rm block} & B_1 & B_3 & B_5 & B_7\\
\hline
\end{array}.\]

We conclude (and prove in Theorem \ref{Torbit}) that the Dieudonn\'e module of ${\rm Jac}(X_{p^3})[p]$ is:
\begin{equation} \label{En=3}
\DD(X_{p^3}) = (\EE/\EE(F^3+V^3))^{r_{3,2}} \oplus (\EE/\EE(F+V))^{g-3r_{3,2}}.
\end{equation}

\subsubsection{The case $n=4$}  See Example \ref{Ecasen=4} for the structure of the Dieudonn\'e module when $n=4$. 

\subsubsection{Strategy for general $n$}

For larger values of $n$ we follow a similar strategy.  
We find a basis of $\Hdr$ using a basis of regular $1$-forms $\omega_{i,j} = x^iy^j \dx$ for $\Hw$ and a basis of functions $f_{i,j}=\frac{1}{x^iy^j} \frac{y^{q-1}}{x}$ for $\Hone$.
We compute the image of $F$ and $V$ on $\Hdr$ and form blocks spanned by basis vectors which have the 
same behavior under iterates of $F$ and $V$.
On each block, either $F$ acts bijectively and $V$ as the zero operator, or vice-versa.
The structure of the Dieudonn\'e module of $X_q$ is determined by the (generalized) permutation of the blocks under $F$ and $V$.

To provide some intuition for the main result, Theorem \ref{Torbit}, 
we discuss in non-precise terms how this structure is related to multiplication-by-$2$ on $\ZZ/(2^n+1)$.  
As in the $n=3$ case, the behavior of $F$ and $V$ is determined by the $p$-adic expansions of $i$ and $j$, specifically whether or not 
the base-$p$ sum of $i$ and $j$ `carries' in the $k$-th digit for $0 \leq k < n$. 
This allows us to index the blocks by binary vectors in $(\ZZ/2)^n$.  Since
Frobenius acts by multiplication-by-$p$ on exponents, it acts
like a `shift' on the base-$p$ digits of $i$ and $j$, and thus by a `shift' on the binary vectors.

We re-index the blocks by non-zero elements of $\ZZ/(2^n+1)$.
Exactly one of $F^{-1}$ and $V$ acts bijectively on each block; it acts like multiplication-by-$2$ on the index.  
In the rest of the paper, we make this description precise, 
thus giving an explicit one-to-one correspondence between the distinct indecomposable factors of the Dieudonn\'e module of $X_q$ 
and the orbits of $\Ttwo$ on $\ZZ/(2^n+1) -\{0\}$.  

\subsection{Some $p$-adic formulae} \label{Spadic}

Given a positive integer $m<p^n$, we fix some notation.  
For $0 \leq h \leq n-1$, let $m_h \in \{0, 1, \ldots, p-1\}$ be the $h$th coefficient in the $p$-adic expansion of $m$:
\[m=m_0 + m_1p + \cdots + m_{n-1} p^{n-1}.\]
For $1 \leq h \leq n$, let
\[
m_{h}^+:=\sum_{l=0}^{h-1} m_lp^{l} \ {\rm and} \ m_{h}^T:=\sum_{l=1}^{h-1} m_lp^{l-1}.
\]
Note that $m=m_0+pm_n^T$ and $m=m_{n-1}p^{n-1}+m_{n-1}^+$ with $0 \leq m_n^T, m_{n-1}^+ \leq p^{n-1}-1$.  
Also
\begin{equation}\label{Ecomparison}
m_h^+ = m_0+pm_{h}^T.
\end{equation}

The following lemma will be useful in the proof of Proposition \ref{Pbinary}.

\begin{lemma} \label{Lpadic}
Suppose $1 \leq i,j \leq p^n$.
\begin{enumerate}
\item If $i_{h}^T+j_{h}^T < p^{h} - 1$ then $i_{h+1}^{+}+ j_{h+1}^{+} < p^{h+1} -1$ and the converse is true if $i_0+j_0 \geq p-1$. 

\item 
If $i_{h+1}^{+} +j_{h+1}^{+} < p^{h+1}-1$ then $(p^{h}-1-i_{h}^T) + (p^{h}-1-j_{h}^T)  \geq p^{h} -1$ and the converse is true if $i_0+j_0 < p-1$.

\item Also: $i_{h}^{+} +j_{h}^{+} < p^h-1$ if and only if
$p-1 + j_{n-1} +p(i_{h}^{+}+j_{h}^{+}) < p^{h+1}-1$.

\item Also: $i_{h}^{+} +j_{h}^{+} < p^h-1$ if and only if
$2p^{h+1}-2-(i_{h}^{+} +j_{h}^{+})p-p-j_{n-1} \geq p^{h+1}-1$.
\end{enumerate}
\end{lemma}

\begin{proof}
\begin{enumerate}
\item The condition $i_{h+1}^{+}+ j_{h+1}^{+} < p^{h+1} -1$ is equivalent to the condition
$(i_{h}^T+j_{h}^T)p < p^{h+1} - (i_0+j_0+1)$.
The result follows since $i_0+j_0+1 \leq 2p-1$ and, under the given condition, $i_0+j_0+1 \geq p$.


\item 
The condition $i_{h+1}^{+} +j_{h+1}^{+} < p^{h+1}-1$ is equivalent to the condition
$(i_{h}^T+j_{h}^T)p < p^{h+1} - (i_0+j_0+1)/p$.
Using the bounds $1 \leq i_0+j_0+1$ and, under the given condition, $i_0+j_0 +1 < p$, 
this condition is equivalent to $i_{h}^T + j_{h}^T \leq p^h-1$, which is equivalent to the 
condition $(p^{h}-1-i_{h}^T) + (p^{h}-1-j_{h}^T)  \geq p^{h} -1$.



\item This follows from the facts that $p(i_{h}^{+} +j_{h}^{+}) \leq p^{h+1}-2p$ when $i_{h}^{+} +j_{h}^{+} < p^h -1$ 
and $p(i_{h}^{+} +j_{h}^{+}) \geq p^{h+1}-p$ when $i_{h}^{+} +j_{h}^{+} \geq p^h -1$ 
and $0 \leq j_{n-1} \leq p-1$.



\item Similar to part (3).


\end{enumerate}

\end{proof}

\section{The de Rham cohomology of Hermitian curves} \label{SFV}

In this section, we compute the actions of $F$ and $V$ with respect to a chosen basis for $\Hdr$.
An essential point is that these actions are scaled permutation matrices with respect to this basis, see Corollary \ref{Cpermutation}.  

\subsection{A basis for the de Rham cohomology} \label{Sbasis}

Consider the following set of lattice points of the plane: 

\[\Delta := \{(i,j) \mid i, j \in \ZZ, \ i,j \geq 0, \ i+j \leq q-2\}.\]

On the Hermitian curve $X_q:y^q+y=x^{q+1}$, the functions $x$ and $y$ have poles at $P_\infty$, 
with $v_{P_\infty}(x)=-q$ and $v_{P_\infty}(y) =-(q+1)$.
Note that $(i,j) \in \Delta$ if and only if $i,j \geq 0$ and $iq+j(q+1) \leq 2g-2$.

\begin{lemma}
A basis for $\Hw$ is given by the set
\[\BB_0 := \{\omega_{i,j}:=x^iy^j \dx \mid (i,j) \in \Delta\}.\]
\end{lemma}

\begin{proof}
This is a special case of \cite[Lemma 1]{sull75}.
\end{proof}

\begin{lemma}
A basis for $\Hone$ is given by the set 
\[\BB_1:=\left\{f_{i,j}:=\frac{1}{x^iy^j} \frac{y^{q-1}}{x} \mid (i,j) \in \Delta\right\}.\]
\end{lemma}

\begin{proof}
To compute $\Hone$, consider the open cover $\cU$ of $X_q$ given by $U_1=X_q \setminus \{P_\infty\}$ 
and $U_2 = X_q \setminus \{(0,y) \mid y^q+y=0\}$.
For $i,j \in \ZZ$, consider the functions $f_{i,j}  \in \Gamma(U_1 \cap U_2, \OO)$.
If $0 \leq j \leq q-1$, the valuation of $f_{i,j}$ at $P_\infty$ is:
\[v_\infty(f_{i,j})=-(q+1)(q-1-j)+q(i+1)=j(q+1)+iq -(q^2+q-1).\]
If also $i+j \leq q-2$, then $v_\infty(f_{i,j})<0$ and so $f_{i,j} \not \in \Gamma(U_2, \OO)$.
If also $i \geq 0$, then $f_{i,j}$ has poles above $x=0$ and so $f_{i,j} \not \in \Gamma(U_1, \OO)$.
Thus (the equivalence class of) the function $f_{i,j}$ is non-zero in $\Hone$ if $i, j \geq 0$ and $i+j \leq q-2$.
These functions $f_{i,j}$ are linearly independent in $\Hone$ since their pole orders at $P_\infty$ are different.
They form a basis for $\Hone$ because there are $g$ pairs $(i,j)$ satisfying these conditions.
\end{proof}

Given $f \in \OO$, it is possible to write $df=\omega(f)_1+\omega(f)_2$ where $\omega(f)_i \in \Gamma(U_i, \Omega^1)$.
Let $\tilde{f}_{i,j} = (f_{i,j}, \omega(f_{i,j})_1, \omega(f_{i,j})_2)$ denote the image of $f_{i,j}$ in $\Hdr$.  

In the rest of this section, we prove that this basis is convenient for computing the actions of $F$ and $V$. 

\begin{corollary} \label{Cpermutation}
With respect to the basis $\BB=\BB_0 \cup \BB_1$, the actions of $V$ and $F$ on $\Hdr$ are scaled permutation matrices, 
i.e., they have at most one non-zero entry in each row and each column. 
\end{corollary}

\begin{proof}
This follows from Lemma \ref{LactionVw}, Proposition \ref{PactionF} and Proposition \ref{PactionVf}.
\end{proof}

\subsection{The action of $V$ on $\Hw$} \label{SVHw}

\begin{lemma} \label{LactionVw}
For $(i,j) \in \Delta$, write $i:=i_0+pi_{n}^T$ and $j:=j_0+pj_{n}^T$ with $0 \leq i_0, j_0 \leq p-1$ and $0 \leq i_n^T, j_n^T \leq p^{n-1}-1$.  
There is a constant $d'_{i,j} \not = 0$ such that the action of $V$ on $\omega_{i,j} \in \Hw$ is given by:
\[
V\left(\omega_{i,j} \right) = 
\begin{cases} 
0 & {\rm if \ } i_0 + j_0 < p -1, \\
d'_{ij}\omega_{p^{n-1}(p-1-i_0) + i_{n}^T, p^{n-1}(i_0+j_0-(p-1)) + j_{n}^T } & {\rm if \ } i_0 + j_0 \geq p -1. 
\end{cases}
\]
\end{lemma}

\begin{proof}
It suffices to computing the image of the Cartier operator $\car$ on $\omega_{i,j}$: 
\begin{align*}
\car(x^iy^j\dx)&= x^{i_{n}^T}y^{j_{n}^T}\car\left(x^{i_0}(x^{q+1}-y^q)^{j_0} \dx \right)\\
&= x^{i_{n}^T}y^{j_{n}^T}\sum_{l=0}^{j_0}\binom{j_0}{l} (-1)^l \car\left( x^{(q+1)(j_0-l)} y^{ql} x^{i_0}\dx \right)\\
&= x^{i_{n}^T}y^{j_{n}^T}\sum_{l=0}^{j_0}\binom{j_0}{l} (-1)^l x^{p^{n-1}(j_0-l)}y^{p^{n-1}l}\car\left( x^{i_0+j_0-l}\dx \right).\\
\end{align*}
Now $\car(x^k \dx) \not = 0$ if and only if $k \equiv -1 \bmod p$.  The exponent of $x$ satisfies
\[
0 \leq i_0+j_0-l \leq 2p-2.
\]
The value congruent to $-1 \mod p$ in this interval is $i_0+j_0-l = p-1$.  Thus $V(\omega_{i,j})=0$ unless $i_0+j_0 \geq p-1$.  If this is the case then substituting $l=i_0+j_0- (p-1)$ gives the desired result
where 
\[
d'_{ij}=\binom{j_0}{i_0+j_0-(p-1)}(-1)^{i_0+j_0-(p-1)}.
\]
\end{proof}

Let $r_{n,i}$ denote the rank of the $i$th iterate of the Cartier operator on $\Hw$ and let $a_n$ be the $a$-number of ${\rm Jac}(X_q)$.
The value of $a_n$ was previously computed in \cite[Proposition 14.10]{Gross}.

\begin{proposition}\label{PrankC}
\begin{enumerate}
\item The rank $r_{n,i}$ of $\car^i$ on $\Hw$ is
\[r_{n,i}=p^n(p+1)^i(p^{n-i}-1)/2^{i+1}.\]
\item The $a$-number $a_n$ of ${\rm Jac}(X_q)$ is
\[a_n=p^n(p^{n-1}+1)(p-1)/4.\]
\end{enumerate}
\end{proposition}

\begin{proof}
Note that $\omega_{i,j} \in {\rm Ker}(\car)$ iff $i_0 + j_0 < p-1$.
More generally, $\omega_{i,j} \in {\rm Ker}(\car^r) - {\rm Ker}(\car^{r-1})$ if and only if:
\[i_0 + j_0 \geq p-1, \ i_1 + j_1 \geq p-1, \ldots, i_{r-2}+j_{r-2} \geq p-1, \ i_{r-1}+j_{r-1} < p-1.\]
This proves the first item.  The second item follows since $a_n=g-r_{n,1}$.
\end{proof}

\subsection{The action of $F$ and $V$ on an image of $\Hone$ in $\Hdr$} \label{SHone}

\subsubsection{The Action of Frobenius} \label{SactionF}

\begin{proposition} \label{PactionF}
For $(i,j) \in \Delta$,  write $i=i_{n-1} p^{n-1} + i_{n-1}^{+}$ and $j= j_{n-1} p^{n-1} + j_{n-1}^{+}$ with $0 \leq i_{n-1}, j_{n-1} \leq p-1$ and $0 \leq  i_{n-1}^{+},  j_{n-1}^{+} \leq p^{n-1}-1$.
Say Case A means that $i_{n-1}^{+} + j_{n-1}^{+} < p^{n-1} -1$ and Case B means that $i_{n-1}^{+} + j_{n-1}^{+} \geq p^{n-1} -1$.
There are constants $c_{i,j}, d_{i,j} \not = 0$ such that the action of $F$ on $\tilde{f}_{i,j} \in \Hdr$ is given by:
\[
F\left(\tilde{f}_{i,j} \right) = 
\begin{cases} 
c_{ij}f_{pi_{n-1}^{+}+(p-1)-i_{n-1},pj_{n-1}^{+}+j_{n-1}+i_{n-1}} & {\rm Case}\  A\\
d_{ij}\omega_{(q-1) - (pi_{n-1}^{+}+(p-1)-i_{n-1}), q-1 - (pj_{n-1}^{+}+j_{n-1}+i_{n-1} +1)} & {\rm Case}\ B.
\end{cases}
\]
\end{proposition}

\begin{proof}
First, 
\begin{align*}
F\left( f_{i,j} \right) &= \frac{1}{y^{j_{n-1}p^n+j_{n-1}^{+}p}x^{i_{n-1}p^n+i_{n-1}^{+}p}} \frac{y^{qp-p}}{x^p} \\
&= \frac{1}{y^{(j_{n-1}^{+}+1)p}x^{(i_{n-1}^{+}+1)p}} \frac{y^{q-1}}{x} \left( y^{q(p-1-j_{n-1})}yx^{-i_{n-1}q+1} \right).
\end{align*}
Let $c_l = (-1)^l \binom{p-1-j_{n-1}}{l}$, then
\[y^{q(p-1-j_{n-1})}yx^{-i_{n-1}q+1}  = \sum_{l=0}^{p-1-j_{n-1}} c_l x^{(q+1)(p-1-j_{n-1}-l)}y^{l+1} x^{-i_{n-1}q+1}.\]
The sum is a linear combination $\sum c_l M_l$ for $0 \leq l \leq p-1-j_{n-1}$ where
\[M_l=x^{q(p-1-j_{n-1}-i_{n-1}-l)}y^{l+1} x^{p-j_{n-1}-l} {\rm \ and \ } c_l = (-1)^l \binom{p-1-j_{n-1}}{l}.\]
For $l \in I_1 = \{0, \ldots, p-2-j_{n-1}-i_{n-1}\}$, the only pole of $M_l$ is at $P_\infty$; then $\sigma_1:=\sum_{l \in I_1} c_l M_1 \in \Gamma(U_1,\OO)$.
For $l \in I_2 = \{p-j_{n-1}-i_{n-1}, \ldots, p-1-j_{n-1}\}$, the only poles of $M_l$ are above $0$; then $\sigma_2:=\sum_{l \in I_2} c_l M_1 \in \Gamma(U_2,\OO)$.

Fix $l^*=p-1-j_{n-1}-i_{n-1}$ and consider the non-zero constants $c_{i,j} :=c_{l^*}$ and $d_{i,j}:=-(j_{n-1}+i_{n-1}+1)c_{l^*}$.
Let
\[\sigma^*:= \frac{1}{y^{(j_{n-1}^{+}+1)p}x^{(i_{n-1}^{+}+1)p}} \frac{y^{q-1}}{x} M_{l^*}=\frac{c_{i,j}}{ y^{pj_{n-1}^{+}+j_{n-1}+i_{n-1}} x^{pi_{n-1}^{+}+p-1-i_{n-1}} } \frac{y^{q-1}}{x}.\]
Consider \[\omega(\sigma^*)_1:=c_{i,j} i_{n-1}^{+} y^{q-1-j_{n-2}^{+}p-j_{n-1}-i_{n-1}} x^{-pi_{n-1}^{+}-p-3+i_{n-1}}\dx,\]
and 
\[\omega(\sigma^*)_2:=d_{i,j} y^{q-1-j_{n-1}^{+}p-j_{n-1}-i_{n-1}-1} x^{q-1 -pi_{n-1}^{+}-p-1+i_{n-1}} \dx.\]
One can check that $\omega(\sigma^*)_i \in \Gamma(U_i, \Omega^1)$ and that $d(\sigma^*)=\omega(\sigma^*)_1 + \omega(\sigma^*)_2$.
Thus $F(\tilde{f}_{i,j}) \equiv (\sigma^*, \omega(\sigma^*)_1, \omega(\sigma^*)_2)$ in $\Hdr$.
In Case A, then $(j_{n-1}^{+}p+j_{n-1}+i_{n-1}) +(pi_{n-1}^{+}+p-1-i_{n-1}) < q-1$.
In this case, $d(\sigma_1) =-\omega(\sigma^*)_1$ and  $d(\sigma_2) =-\omega(\sigma^*)_2$.  Taking the quotient by $\sigma_1$  and $\sigma_2$ yields that
\[F\left(\tilde{f}_{i,j} \right)=c_{i,j}f_{pi_{n-1}^{+}+(p-1)-i_{n-1},pj_{n-1}^{+}+j_{n-1}+i_{n-1}}.\]
In Case B, then $\omega(\sigma^*)_1$ is regular.  
In this case, $d(\sigma_2+\sigma^*) = \omega(\sigma^*)_1 = -d(\sigma_2)$.  
Taking the quotient by $\sigma_1$ and  $\sigma^*+\sigma_2$ yields that
\[F\left(\tilde{f}_{i,j} \right)=d_{i,j} \omega_{(q-1) - (pi_{n-1}^{+}+(p-1)-i_{n-1}), q-1 - (pj_{n-1}^{+}+j_{n-1}+i_{n-1} +1)}.\]
\end{proof}

\subsubsection{The Action of Verschiebung} \label{SactionV}

\begin{proposition} \label{PactionVf}
For $(i,j) \in \Delta$, write $i=i_0 + i_{n}^T p$ and $j= j_0 + j_{n}^Tp$ with $0 \leq i_0, j_0 \leq p-1$ and $0 \leq i_n^T, j_n^T \leq p^{n-1}-1$. 
Let $i^*=p^{n-1}i_0 + (p^{n-1} -1 - i_{n}^T)$ and $j^*=p^{n-1}(p-2-i_0-j_0) + (p^{n-1}-1-j_{n}^T)$.
There is a constant $c'_{i,j} \not = 0$ such that the action of $V$ on $\tilde{f}_{i,j} \in \Hdr$ is given by:
\[
V\left(\tilde{f}_{i,j} \right) = 
\begin{cases} 
c'_{i,j}\omega_{i^*, j^*} & {\rm if \ } i_0 + j_0 < p -1 \\
0 & {\rm if \ } i_0 + j_0 \geq p -1.
\end{cases}
\]
\end{proposition}

\begin{proof}
Let 
\[
\omega(f_{i,j})_1= - (i+1)y^{q-j-1}x^{-i-2} \dx {\rm \ and \ }  \omega(f_{i,j})_2=-(j+1)y^{q-j-2}x^{-i-1}\dy
\]
One can check that $\omega(f_{i,j})_1 \in \Gamma(U_1, \Omega^1)$ and $\omega(f_{i,j})_2 \in \Gamma(U_2, \Omega^1)$
and that $df_{i,j}=\omega(f_{i,j})_1 + \omega(f_{i,j})_2$.

Recall that $V(f, \omega) := (0, \car(\omega))$.
Since $\car(\omega(f_{i,j})_1)+\car(\omega(f_{i,j})_2)=0$, it is only necessary to compute $\car(- \omega(f_{i,j})_1)$ which equals
\[\car\left((i+1)y^{q-j-1}x^{-i-2}\dx\right) = (i_0+1)y^{q/p - j_{n}^T-1}x^{-i_{n}^T}\car\left(y^{p-j_0-1}x^{-i_0-2}\dx\right).\]
Now, $\car\left(y^{p-j_0-1}x^{-i_0-2}\dx\right) = \car\left(\left(x^{q+1} - y^q\right)^{p-j_0-1}x^{-i_0-2}\dx\right)$
which equals
\[\sum_{l=0}^{p-1-j_0}\binom{p-1-j_0}{l} \car \left( x^{(q+1)(p-1-j_0-l)} (-y)^{ql}x^{-i_0-2}\dx\right).\]
Note that 
\[\car \left( x^{(q+1)(p-1-j_0-l)} (-y)^{ql}x^{-i_0-2}\dx\right)= (-1)^l x^{p^{n-1}(p-1-j_0-l)}y^{p^{n-1}l} \car\left(x^{p-3-j_0-i_0-l}\dx\right).\]
The exponent $e=p-3-j_0-i_0-l$ of $x$ satisfies
\[-p -1 \leq -i_0 -2 = p-3-j_0-i_0-(p-1-j_0) \leq e \leq p-3.\]
Recall that $\car(x^e \dx) \not = 0$ if and only if $e \equiv -1 \bmod p$.  
Note that $e=-p-1$ only when $i_0=p-1$, in which case the term is trivialized by $\car$ as seen above.
As such, the only term which is not trivialized by $\car$ is when $e=-1$, i.e., when
\[
l=p-2-i_0-j_0.
\]
Thus $V(\tilde{f}_{i,j})=0$ if $i_0 + j_0 \geq p-1$.  
If $i_0 + j_0 \leq p-2$, the claimed result follows by substituting $l=p-2-i_0-j_0$
and using the non-zero constant
\[
c'_{i,j}=(i_0+1)\binom{p-1-j_0}{p-2-j_0-i_0}(-1)^{p-2-i_0-j_0}.
\]
\end{proof}

\section{Decomposition of the de Rham cohomology of Hermitian curves} \label{SHdr}

This is the main result of this section:

\begin{corollary} \label{Cmult2}
There is a decomposition $\Hdr = \oplus_{1 \leq t \leq 2^n}{B_t}$ such that the morphisms
$V$ and $F^{-1}$ act on the blocks $B_t$ by multiplication-by-2 on the indices modulo $2^n+1$ as follows.  

If $2^{n-1}+1 \leq t \leq 2^n$, then there is an isomorphism $V:B_t \to B_{2t \bmod 2^n+1}$.
  
If $1 \leq t \leq 2^{n-1}$, then $B_t \subset {\rm ker}(V)={\rm Im}(F)$ and there is an isomorphism $F^{-1}:B_t \to B_{2t}$. 
\end{corollary}

In order to prove this, we partition the basis $\BB=\BB_0 \cup \BB_1$ for $\Hdr$ into $2^n$ sets which are well-suited for studying the action of $F$ and $V$.
The sets are first indexed by vectors $\vec{b} \in (\ZZ/2)^n$ and then by non-zero $t \in \ZZ/(2^n+1)$. 

\subsection{A binary vector decomposition} \label{Sbvd}
Given $i, j \geq 0$ such that $0 \leq i+j \leq q-2$, 
recall the definitions of $i_k^{+}, j_k^{+}, i_{k}^T, j_{k}^T$ from Section \ref{Spadic}.
For $0 \leq h \leq n-2$, let 
\[b_h(i,j) = 
\begin{cases}
0 & {\rm if } \ i_{h+1}^{+}+ j_{h+1}^{+} < p^{h+1} -1,\\
1 & {\rm otherwise}.
\end{cases}
\]
For example, $b_0(i,j)=0$ when $i_0 + j_0 < p-1$ and $b_1(i,j)=0$ when $i_0+i_1p + j_0 +j_1p < p^2-1$.

\begin{definition} \label{Dblock}
For each element of the basis $\BB$ for $\Hdr$, define a vector $\vec{b}=(b_0, \ldots, b_{n-1}) \in  (\ZZ/2)^n$ as follows:
If $\tilde{f}_{i,j} \in \BB \cap \Hone$, let $b_{n-1}(i,j)=0$ and  
\[\vec{b}(\tilde{f}_{i,j})=(b_0(i,j), \ldots, b_{n-2}(i,j), 0).\]
If $\omega_{i,j} \in \BB \cap \Hw$, let $b_{n-1}(i,j)=1$ and
\[\vec{b}(\omega_{i,j})=(b_0(i,j), \ldots, b_{n-2}(i,j), 1).\]
Finally, for $\vec{b} \in (\ZZ/2)^n$, consider the subspace
\[\Hdr_{\vec{b}}:=\Span\{\lambda \in \BB \ \mid \ \vec{\lambda}=\vec{b}\}.\]
For notational purposes, let $\Hdr_{0}=0$.  
\end{definition}

\begin{lemma} \label{rem:countingblocks}
Given a vector $\vec{b}=(b_0, \ldots, b_{n-1}) \in (\ZZ/2)^n$, 
let $n_s$ (resp.\ $n_d$) be the number of adjacent terms of 
$(b_0,\ldots,b_{n-2})$
which are equal (resp.\ different).
Then 
\[
\dim(\Hdr_{\vec{b}}) = \left( \frac{p(p+1)}{2} \right)^{n_s+1+b_0-b_{n-2}} \left( \frac{p(p-1)}{2} \right)^{n_d+1+b_{n-2}-b_0}.
\]
\end{lemma}

\begin{proof}
The values $b_k(i,j)$ are determined by the behavior of the base-$p$ expansion of the sum $i+j+1$.
Namely, $b_k(i,j)=1$ if and only if the sum $i+j+1$  `carries' in the $k$-th digit.
Since $i+j < q-1$, there is no `carrying' out of the last digit; 
the addition of $1$ can be thought of as `carrying' into the first digit. 
Then $\dime(\Hdr_{\vec{b}})$ is the number of pairs $(i,j)$ satisfying the  `carrying pattern' associated to $\vec{b}$.
It equals the product of the numbers $\alpha_k$ of pairs of $p$-adic digits $(i_k, j_k)$ as $0 \leq k \leq n-1$, where 
 $\alpha_k= \#\{(i_k, j_k) \mid  0 \leq i_k,j_k \leq p-1, \  i_k+j_k \leq p-1-|b_k-b_{k-1}|\}$.
\end{proof}

\subsection{A congruence decomposition} \label{Scongdec}

To index blocks with integers instead of binary vectors, consider this bijection $T:(\ZZ/2)^n \to \ZZ/(2^n+1)-\{0\}$.

\begin{definition}
Given $\vec{b}=(b_0, \ldots, b_{n-1}) \in (\ZZ/2)^n$:
\begin{enumerate}
\item if $b_{n-1}=1$, let $T(\vec{b})=2^{n-1}b_0 + \cdots 2b_{n-2} + 1$;
\item if $b_{n-1}=0$, let $T(\vec{b})=2^n-(2^{n-1}b_0 + \cdots 2b_{n-2})$.
\end{enumerate}
\end{definition}

When $r$ is even (resp.\ odd), the coordinates of the vector $T^{-1}(r)$ are the coefficients of the binary expansion of $r$ (resp.\ 
written in reverse order). 

\subsection{Block structure} \label{Sblockstructure}

Consider the decomposition $\Hdr = \oplus_{1 \leq t \leq 2^n}{B_t}$ where $B_t := \Span\{\lambda \in \BB \ \mid \ T(\vec{\lambda})=t\}$ for $1 \leq t \leq 2^n$.
Corollary \ref{Cmult2} is an immediate consequence of the next result.

\begin{theorem} \label{TblockactionFV}
The actions of $V$ and $F$ on $\Hdr$ satisfy the following:
\begin{enumerate}
\item if $1 \leq t \leq 2^{n-1}$, then $V(B_t) = 0$; 
\item if $2^{n-1} + 1 \leq t \leq 2^n$, then there is an isomorphism $V \mid_{B_t}: B_t \to B_{2t-2^n-1}$;
\item if $t$ is odd, then $F(B_t)=0$;
\item if $t$ is even, then there is an isomorphism $F \mid_{B_t}:B_t \to B_{t/2}$. 
\end{enumerate}
\end{theorem}

The proof of Theorem \ref{TblockactionFV} occupies the rest of the section.

\subsection{The action of $F$ and $V$ in terms of binary vectors}

In this section, we show that $F$ and $V$ act on $\Hdr$ by permuting the subspaces $\Hdr_{\vec{b}}$ for $\vec{b} \in (\ZZ/2)^n$.
The next definition summarizes the change in the binary vector under the action of $F$ and $V$. 

\begin{definition}\label{defn:FVbin} Let $\iota$ be the transposition $(0,1)$.
Given $\vec{b}=(b_0, \ldots, b_{n-1})$, define $\vec{Vb}$ and $\vec{Fb}$ as follows: 
\begin{enumerate}
\item Action of $V$ on $\Hw$: If $b_{n-1}=1$ and $b_0=0$, let $\vec{Vb}=0$.\\
If $b_{n-1}=1$ and $b_0=1$, let $\vec{Vb}=(b_1, \ldots, b_{n-2}, 0, 1)$, {\it (left shift with flip in last two positions).}

\item Action of $V$ on $\Hone$: If $b_{n-1}=0$ and $b_0=1$, let $\vec{Vb}=0$.\\
If $b_{n-1}=0$ and $b_0=0$, let $\vec{Vb}=(\iota(b_1), \ldots, \iota(b_{n-2}), 1,1)$, {\it (left shift with flip in all positions).}

\item Action of $F$ on $\Hw$: If $b_{n-1}=1$, let $\vec{Fb}=0$.

\item Action of $F$ on $\Hone$: 
\begin{description}
\item{[A]} If $b_{n-1}=0$ and $b_{n-2}=0$, let $\vec{Fb}=(1, b_0, \ldots, b_{n-3}, 0)$, {\it (right shift with flip in first position).}
\item{[B]} If $b_{n-1}=0$ and $b_{n-2}=1$, let $\vec{Fb}=(0, \iota(b_0), \ldots, \iota(b_{n-3}), 1)$, {\it (right shift with flip in all interior positions).}
\end{description}
\end{enumerate}
\end{definition}

\begin{proposition} \label{Pbinary}
For each binary vector $\vec{b} \in (\ZZ/2)^n$:
\[V\Hdr_{\vec{b}} \cong \Hdr_{\vec{Vb}} \ {\rm and } \ F\Hdr_{\vec{b}} \cong \Hdr_{\vec{Fb}}.\]
\end{proposition}

\begin{proof}
The proof that the image of $F$ or $V$ is in the claimed block is divided into cases as in Definition \ref{defn:FVbin}. 
\begin{enumerate}
\item Action of $V$ on $\Hw$: 
If $\omega_{i,j} \in \Hdr_{\vec{b}}$, the claim is that $V(\omega_{i,j}) \in \Hdr_{\vec{Vb}}$.
Note that $b_{n-1}(\omega_{i,j})=1$ by definition.  
If $b_0(\omega_{i,j})=0$ then $V(\omega_{i,j})=0$ by Lemma \ref{LactionVw}.

Suppose $b_0(\omega_{i,j}) =1$, i.e., $i_0+j_0 \geq p-1$.
By Definition \ref{defn:FVbin}(1), it suffices to show that $b_{k-1}(V(\omega_{i,j})) = b_k(\omega_{i,j})$ for $k \in \{1 \ldots n-1\}$.  
By definition, $b_k(\omega_{i,j}) =0$ if and only if $i_{k+1}^{+}+ j_{k+1}^{+} < p^{k+1} -1$.
By Lemma \ref{Lpadic}(1), since $i_0+j_0 \geq p-1$, this is equivalent to
$i_{k}^T+j_{k}^T < p^{k} - 1$.
By Lemma \ref{LactionVw}, this is equivalent to $b_{k-1}(V(\omega_{i,j})) =0$.
In particular, $b_{n-2}(V(\omega_{i,j})) = 0$ since $i+j<p^n-1$.

\item Action of $V$ on $\Hone$:
If $\tilde{f}_{i,j} \in \Hdr_{\vec{b}}$, the claim is that $V(\tilde{f}_{i,j}) \in \Hdr_{\vec{Vb}}$.
Note that  $b_{n-1}(\tilde{f}_{i,j})=0$ by definition. 
If $b_0(\tilde{f}_{i,j})=1$ then $V(\tilde{f}_{i,j})=0$ by Proposition \ref{PactionVf}.

Suppose $b_0(\tilde{f}_{i,j}) =0$, i.e., $i_0+j_0 < p-1$.
By Definition \ref{defn:FVbin}(2), it suffices to show $b_{k}(\tilde{f}_{i,j}) =0$ if and only if $b_{k-1}(V(\tilde{f}_{i,j})) =1$ for $1 \leq k \leq n-1$.
By definition, $b_{h}(\tilde{f}_{i,j}) =0$ means that $i_{h+1}^{+} +j_{h+1}^{+} < p^{h+1}-1$. 
By Lemma \ref{Lpadic}(2), this is equivalent to
$(p^{k}-1-i_{k}^T) + (p^{k}-1-j_{k}^T) \geq p^{k} -1$.
This is equivalent to $b_{k-1}(V(\tilde{f}_{i,j})) =1$ by Proposition \ref{PactionVf}.  
In particular, $b_{n-2}(V(\tilde{f}_{i,j})) =1$ since $b_{n-1}(\tilde{f}_{i,j}) =0$.

\item Action of $F$ on $\Hw$:  If $\omega_{i,j} \in \Hdr_{\vec{b}}$, then $F(\omega_{i,j})=0$ by Section \ref{sub:fv}

\item Action of $F$ on $\Hone$:

For [A], given $\tilde{f}_{i,j} \in \Hdr_{\vec{b}}$ such that $F(\tilde{f}_{i,j}) \in \Hone$, 
the claim is that $F(\tilde{f}_{i,j}) \in \Hdr_{\vec{Fb}}$.
By Proposition \ref{PactionF}, $F(\tilde{f}_{i,j}) \in \Hone$ when $b_{n-2}(\tilde{f}_{i,j}) =0$.
By Definition \ref{defn:FVbin}(3), it suffices to show 
$b_h(F(\tilde{f}_{i,j})) = b_{h-1}(\tilde{f}_{i,j})$ for $1 \leq h \leq n-1$.
By definition, $b_{h-1}(\tilde{f}_{i,j}) =0$ if and only if $i_h^+ + j_h^+ < p ^h -1$.
By Lemma \ref{Lpadic}(3), this is equivalent to
$p-1 + j_{n-1} +p(i_{h}^{+}+j_{h}^{+}) < p^{h+1}-1$.
By Proposition \ref{PactionF}[A], this is equivalent to $b_h(F(\tilde{f}_{i,j}))=0$.
Also notice that $b_0(F(\tilde{f}_{i,j}))=1$ since $p-1+j_{n-1} \geq p-1$.

For [B], given $\tilde{f}_{i,j} \in \Hdr_{\vec{b}}$ such that $F(\tilde{f}_{i,j}) \in \Hw$, 
the claim is that $F(\tilde{f}_{i,j}) \in \Hdr_{\vec{Fb}}$.
By Proposition \ref{PactionF}, $F(\tilde{f}_{i,j}) \in \Hw$ when $b_{n-2}(\tilde{f}_{i,j}) =1$.
By Definition \ref{defn:FVbin}(4), it suffices to show 
$b_{k-1}(\tilde{f}_{i,j})=0$ if and only if $b_k(F(\tilde{f}_{i,j})) =1$ for $1 \leq k \leq n-1$.
By definition, $b_{k-1}(\tilde{f}_{i,j})=0$ if and only if $i_k^+ + j_k^+ < p^k-1$.
By Lemma \ref{Lpadic}(4), this is equivalent to
$2p^{k+1}-2-(i_{k}^{+}+j_{k}^{+})p -p -j_{n-1} \geq p^{k+1}-1$.
By Proposition \ref{PactionF}[B], 
this is equivalent to $b_k(F(\tilde{f}_{i,j})) =1$.
Also note that $b_0(F(\tilde{f}_{i,j}))= 0 $ since $p-2-j_{n-1} < p-1$.

\end{enumerate}

Here is a sketch of 3 ways to prove that $F$ or $V$ surjects onto the claimed block.  
The first method is to compute an explicit pre-image in $\Hdr_{\vec{b}}$ for a given element of $\Hdr_{\vec{Fb}}$ or $\Hdr_{\vec{Vb}}$.  We omit this calculation.
The second method is to prove that the blocks $\Hdr_{\vec{b}}$ are irreducible $\FF_{q^2}[G]$-modules using \cite[4.7]{hj90}.
The third method is to use Corollary \ref{Cpermutation} to show that $F$ and $V$ either trivialize or act injectively on a block; in the latter case, 
the action must also be surjective by a dimension count from Lemma \ref{rem:countingblocks}.
\end{proof}

\begin{proof} {\it Proof of Theorem \ref{TblockactionFV}. }
Suppose $\vec{b} \in (\ZZ/2)^n$ is such that $T(\vec{b})=t$.
\begin{enumerate}
\item If $T(\vec{b}) \leq 2^{n-1}$, then either $b_{n-1}=1$ and $b_0=0$, or $b_{n-1}=0$ and $b_0=1$.  
Then $V\Hdr_{\vec{b}}=0$ by Lemma \ref{LactionVw} in the former case and by Proposition \ref{PactionVf} in the latter case.

\item If $T(\vec{b}) > 2^{n-1}$, then either $b_{n-1}=1$ and $b_0=1$, or $b_{n-1}=0$ and $b_0=0$.  
In the former case, by Definition \ref{defn:FVbin}(1) and Proposition \ref{Pbinary}(1), 
\begin{align*}
T(V\vec{b}) &= 2^{n-1}b_1 + \cdots + 2^2 b_{n-2} + 1\\
&=2(2^{n-1} + 2^{n-1}b_1 + \cdots + 2 b_{n-2} +1) -2^n-1=2t-(2^n+1).
\end{align*} 
In the latter case, by Definition \ref{defn:FVbin}(2) and Proposition \ref{Pbinary}(2),
\begin{align*}
T(V\vec{b}) &= 2^{n-1}(1-b_1) + \cdots + 2^2 (1-b_{n-2}) +2 +1\\ 
& =2(2^{n} - 2^{n-1}b_0 - \ldots - 2 b_{n-2})  -2^n-1=2t-(2^n+1).
\end{align*} 


\item If $T(\vec{b})$ is odd, then $b_{n-1} = 1$ and $B_t \subset \Hw$.  Then $F(B_t) =0$ by Proposition \ref{Pbinary}(3).

\item Suppose $T(\vec{b})$ is even.  If $b_{n-2}=0$, then Proposition \ref{Pbinary}(4)[A] implies that
\[T(F\vec{b}) =2^n - (2^n + 2^{n-1} + 2^{n-2}b_0 + \ldots -2b_{n-3})=t/2.\]
If $b_{n-2}=1$, then Propositon \ref{Pbinary}(4)[B] implies that
 \begin{align*}
T(F\vec{b})&=2^{n-2}(1-b_0) + 2^{n-3} (1-b_1) + \ldots  + 2(1-b_{n-3}) +1 \\
&=2^{n-1} - 2^{n-2}b_0 - \ldots -2b_{n-3} - b_{n-2} = t/2.
\end{align*}
\end{enumerate}
\end{proof}

\section{The Dieudonn\'e modules of the Hermitian curves} \label{Sthmorbit}

In this section, we prove Theorem \ref{Torbit} which determines the structure of the $p$-torsion group scheme ${\rm Jac}(X_q)[p]$
for all primes $p$ and $n \in \NN$. 
The result is phrased in terms of the Dieudonn\'e module, which we denote by
\[\DD(X_{p^n}):=\DD({\rm Jac}(X_{p^n})[p]).\] 
Specifically, we prove that the distinct indecomposable factors of $\DD(X_{p^n})$ 
are in bijection with orbits of $\ZZ/(2^n+1) - \{0\}$ under $\Ttwo$ and compute the multiplicity of each factor.
In Section \ref{Sstructure2}, we explain how the structure of each indecomposable factor is determined from the combinatorics of the orbit.
From this, one can compute the Ekedahl-Oort type of ${\rm Jac}(X_q)[p]$ in any specific case but it is hard (and non-illuminating) to find formulae in general.

\subsection{Combinatorial properties of orbits} \label{Sstructure}

Two elements $s,t \in \ZZ/(2^n+1) -\{0\}$ are in the same {\it orbit} under $\Ttwo$
if and only if $2^i s \equiv t \bmod 2^n+1$ for some $i \in \ZZ$. 
Every orbit $\sigma$ of $\ZZ/(2^n+1) -\{0\}$ under $\Ttwo$ is {\it symmetric} in that $(-1)\sigma = \sigma$, 
because $2^n \equiv -1 \bmod 2^n+1$.

\begin{definition} \label{D1}
Let $\sigma=(\sigma_1 \ldots, \sigma_r)$ be an orbit of $\ZZ/(2^n+1) -\{0\}$ under $\Ttwo$.  Let $\sigma_0=\sigma_r$.
\begin{enumerate}
\item The {\it length} $|\sigma|$ of $\sigma$ is $r$.
\item An entry $\sigma_i \in \sigma$ is a {\it local maximum} if $\sigma_{i-1} < \sigma_i > \sigma_{i+1}$.
and is a {\it local minimum} if $\sigma_{i-1} > \sigma_{i} < \sigma_{i+1}$.
Let ${\rm Max}(\sigma)$ (resp.\ ${\rm Min}(\sigma)$) be the set of local maximums (resp.\ minimums) of $\sigma$.


\item The {\it $a$-number} of $\sigma$ is $a(\sigma)=\#{\rm Max}(\sigma)=\#{\rm Min}(\sigma)$. 

\end{enumerate}
\end{definition}

\begin{lemma}
If $\sigma$ is an orbit of $\ZZ/(2^n+1) -\{0\}$ under $\Ttwo$, then $|\sigma|$ is even and $a(\sigma)$ is odd.
\end{lemma}

\begin{proof}
The length is even since $\sigma$ is symmetric under $-1$.

Without loss of generality, suppose $\sigma_1 = {\rm min}\{\sigma_i \in \sigma\}$.
Since $\sigma$ is symmetric under $-1$, the absolute maximum 
of the entries in $\sigma$ is $\sigma_{\frac{r}{2}+1}$.  More generally, $\sigma_{1+i} \equiv -\sigma_{\frac{r}{2}+i} \bmod \ZZ/(2^n+1)$.  
Thus $\sigma$ can be divided into two parts, termed the left half and the right half.
 
Consider the number of local minimums and local maximums in $\sigma$, excluding $\sigma_1$ and $\sigma_{\frac{r}{2}+1}$.
On each half, the number of local minimums equals the number of local maximums, 
by an increasing/decreasing argument.
By symmetry, the number of local minimums in the left half equals the number of local maximums in the right half.  It follows that the number of local maximums other than $\sigma_{\frac{r}{2}+1}$ is even, so $a(\sigma)$ is odd.
\end{proof}

The next definition measures the distances between the local maximums and minimums of $\sigma$. 

\begin{definition} \label{D2}
\begin{enumerate}
\item If $\sigma_i \in {\rm Min}(\sigma)$, the {\it left distance} of $\sigma_i$ is $\ell(\sigma_i)={\rm min}\{j \in \NN \mid \sigma_{i-j} \in {\rm Max}(\sigma)\}$; \\
and the {\it right distance} of $\sigma_i$ is $\rho(\sigma_i)={\rm min}\{j \in \NN \mid \sigma_{i+j} \in {\rm Max}(\sigma)\}$.

\item If $\sigma_i \in {\rm Min}(\sigma)$, the {\it left parent} of ${\sigma_i}$ is $L(\sigma_i)$ where $L(\sigma_i):=\sigma_{i-\ell(\sigma_i)}$;\\ 
and the {\it right parent} of $\sigma_i$ is $R(\sigma_i)$ where $R(\sigma_i):=\sigma_{i+\rho(\sigma_i)}$.  

\end{enumerate}

\end{definition}

\begin{remark}
The structure of an orbit is determined by the binary expansion of its minimal element,  see Proposition \ref{Pdistinct}. 
The symmetric property of the orbits can be used to show that the number of orbits of length $2n$ is the number of binary self-reciprocal polynomials of degree $2n$; 
which is found in sequence A000048 in the Online Encyclopedia of Integer Sequences \cite{OEIS}.  
The total number of orbits is found in sequence A000016 in \cite{OEIS}. 
\end{remark}

\subsubsection{Short orbits}

Most orbits of $\ZZ/(2^n+1) -\{0\}$ under $\Ttwo$ have maximum length $2n$.
The following results about short orbits are used in Proposition \ref{Pdistinct}, Corollary \ref{Canumber1} and Applications \ref{Aellipticrank} and \ref{ASelmer}.

\begin{lemma} \label{Lup}
Suppose $n=ck$ for $k \in \NN$ odd and let $L=(2^n+1)/(2^c+1)$.
The multiplication-by-$L$ group homomorphism $\ZZ/(2^c+1) \hookrightarrow \ZZ/(2^n+1)$, given by
$\alpha \mapsto L \alpha$, induces a bijection \[\beta:\sigma \mapsto \sigma_L\]
between orbits $\sigma$ of $\ZZ/(2^c+1) -\{0\}$ under $\Ttwo$ and
orbits $\sigma_L$ of $\langle L \rangle  \cap (\ZZ/(2^n+1) - \{0\})$ under $\Ttwo$.
\end{lemma}

\begin{proof}
Omitted.
\end{proof}

\begin{lemma} \label{Ldown}
Suppose $\hat{\sigma}$ is an orbit of $\ZZ/(2^n+1)-\{0\}$ under $\Ttwo$ with $|\hat{\sigma}| < 2n$.  
Then $n=ck$ for some $k \in \NN$ odd and $\hat{\sigma}=\sigma_L$ 
for some orbit $\sigma$ of $\ZZ/(2^c+1) -\{0\}$ under $\Ttwo$.
\end{lemma}

\begin{proof}
Let $\hat{\sigma}$ be an orbit of length $2c$ where $c<n$.
Without loss of generality, suppose $\sigma_1={\rm min}\{\sigma_i \in \hat{\sigma}\}$. Let $L = \gcd(\sigma_1, 2^{n-1})$ and write $\sigma_1 = LM$.
Let $M^{-1}$ be the inverse of $M$ modulo $2^n+1$.  Then
$
\sigma_{M^{-1}}=(L,2L,\ldots,2^cL,-L,-2L,\ldots,-2^cL)
$
is another orbit of $\ZZ/(2^n+1) -\{0\}$ under $\Ttwo$ with length $2c$ and $a$-number 1.  The sequence $L,2L,\ldots,2^cL$ is strictly increasing and $2^cL < 2^n+1$.
Now, $c$ is the smallest positive integer such that $2^c L \equiv - L \bmod 2^n+1$. Thus $(2^c+1)L =m(2^n+1)$ for some $m \in \ZZ$.  However, 
The fact that $L < (2^n+1)/2^c$ implies that $(2^c+1)L=2^n+1$ and so $n=ck$ for some $k \in \NN$ odd.
Let $\sigma = \frac{1}{L} \hat{\sigma}:=(\frac{\sigma_1}{L}, \ldots, \frac{\sigma_r}{L})$.
Then $\sigma$ is an orbit of $\ZZ/(2^c+1) -\{0\}$ under $\Ttwo$
and $\hat{\sigma}=\sigma_L$.
\end{proof}

\subsection{The construction of a Dieudonn\'e module for each orbit} \label{Sstructure2}

We define a Dieudonn\'e module $\DD(\sigma)$ for every orbit $\sigma$ of $\ZZ/(2^n+1) -\{0\}$ under $\Ttwo$ in terms of generators and relations.  In the next subsection we prove that these modules are in fact the indecomposable factors of the Dieudonne\'e module of $X_{p^n}$.

For convenience, we replace an entry $\sigma_i \in \sigma$ by a variable $B_{\sigma_i}$. 
If $\sigma_i \in {\rm Max}(\sigma)$, then $B_{\sigma_i}$ is a {\it generator block}.
If $\sigma_i \in {\rm Min}(\sigma)$, then $B_{\sigma_i}$ is a {\it relation block}.

\begin{definition} \label{D3}
Let $\sigma=(\sigma_1 \ldots, \sigma_r)$ be an orbit of $\ZZ/(2^n+1) -\{0\}$ under $\Ttwo$.
The Dieudonn\'e module $\DD(\sigma)$ is the quotient of the left $\EE$-module generated by variables
\[\{B_{\sigma_i} \mid \sigma_i \in {\rm Max}(\sigma)\},\]
by the left ideal of relations generated by
\[\{V^{\ell(\sigma_i)} B_{L(\sigma_i)} + F^{\rho(\sigma_i)} B_{R(\sigma_i)} = 0 \mid  {\rm \ for \ all \ }\sigma_i \in {\rm Min}(\sigma)\}.\]
\end{definition}

The following diagram illustrates the definition.
\[
\xymatrix{
B_{L(\sigma_i)} \ar@{->}[rd]_{V^{\ell(\sigma_i)}} & & B_{R(\sigma_i)} \ar@{->}[ld]^{F^{\rho(\sigma_i)}} \\ 
& B_{\sigma_i} & \\
}
\]

\begin{example} \label{Eorbit1}
The orbit of $1$ in $\ZZ/(2^n+1) -\{0\}$ under $\Ttwo$ is $\sigma = (1,2,\ldots, 2^n, 2^n-1, \ldots, 2^{n-1}+1)$.
It has $a(\sigma)=1$.
The generator block is $B_{2^n}$.  The relation block is $B_1$.
Also $\ell(\sigma_1)=\rho(\sigma_1)=n$.
Thus 
\[\DD(\sigma) \simeq \EE /\EE(F^n+V^n).\]
This is the Dieudonn\'e module of the unique symmetric  ${\rm BT}_1$ group scheme of rank $p^{2n}$ having $p$-rank $0$ and $a$-number $1$.
This group scheme, which we denote by $I_{n,1}$, has Ekedahl-Oort type $[0,1,2, \ldots, n-1]$; see \cite[Lemma 3.1]{Pr:sg} for details.
\end{example}

\begin{example} \label{EI43}
When $n=4$, an orbit of $\Ttwo$ on $\ZZ/17$ is $\sigma=\{3,6,12,7,14,11,5,10\}$ as illustrated below.
{\tiny \[
\xymatrix{
 & & B_{12} \ar@{->}[rd]^{V} & & B_{14} \ar@{->}[rd]^{V} & & & &  \\
 & B_6 \ar@{->}[ru]^{F^{-1}} & & B_7 \ar@{->}[ru]^{F^{-1}} & & B_{11} \ar@{->}[rd]^{V} & & B_{10} \ar@{->}[rd]^{V}& \\
B_3 \ar@{->}[ru]^{F^{-1}} & & & & & & B_5 \ar@{->}[ru]^{F^{-1}} & & B_3 \\ 
}
\]}
It has $a(\sigma)=3$.
The generator blocks are $B_{12}$, $B_{14}$ and $B_{10}$ and the relation blocks are $B_3$, $B_7$, and $B_5$.
The relations are $FB_{14}+VB_{12}=0$ and $FB_{10}+V^2B_{14}=0$ and $F^{2}B_{12}+VB_{10}=0$.
Thus 
\[\DD(\sigma) = (\EE B_{12} \oplus \EE B_{14} \oplus \EE B_{10})/\EE(FB_{14}+VB_{12}, FB_{10}+V^2B_{14}, F^2B_{12}+VB_{10}).\]
Then $\DD(\sigma) \simeq \DD(I_{4,3})$ where $I_{4,3}$ is the rank 8 ${\rm BT}_1$ with Ekedahl-Oort type $[0,0,1,1]$ \cite[Remark 5.13]{EP}. 
\end{example}

\begin{lemma} \label{LDsigma}
The left $\EE$-module $\DD(\sigma)$ is symmetric, is trivialized by both $F$ and $V$, has dimension $|\sigma|$, and has $a$-number $a(\sigma)$.
\end{lemma}

\begin{proof}
First, $\DD(\sigma)$ is symmetric since $\sigma$ is symmetric.
Second, the relations $FV=VF=0$ imply that $V^{\ell(\sigma_i)+1}B_{L(\sigma_i)}=0$ and 
$F^{\rho(\sigma_i)+1}B_{R(\sigma_i)}=0$ for each $\sigma_i \in {\rm Min}(\sigma)$.
Since every generator block is both a left and a right parent, powers of $F$ and $V$ trivialize all the 
generator blocks.
Third, the dimension equals the number of distinct images of the generator blocks under powers of $F$ and of $V$, which is exactly $|\sigma|$.
Finally, the $a$-number equals the number of generators as an $\EE$-module.
\end{proof}

\begin{proposition} \label{Pdistinct}
If $\sigma'$ and $\sigma$ are distinct orbits of $\ZZ/(2^n+1) -\{0\}$ under $\Ttwo$, then $\DD(\sigma) \not \simeq \DD(\sigma')$.
\end{proposition}

\begin{proof}
By Lemma \ref{LDsigma}(3), the structure of $\DD(\sigma)$ determines $|\sigma|$.
The bijection $\beta$ in Lemma \ref{Lup} preserves the $\EE$-module structure of the Dieudonn\'e module: $\DD(\sigma_L) \simeq \DD(\sigma)$.
By Lemmas \ref{Lup} and \ref{Ldown}, it suffices to restrict to the case $|\sigma|=2n$.
Without loss of generality, suppose $\sigma_1={\rm min}\{\sigma_i \in \sigma\}$. 
By minimality, $\sigma_1 < 2^{n-1}$ (otherwise $-\sigma_1 < \sigma_1$) and $\sigma_1$ is odd.  
Notice that  $\sigma_i > \sigma_{i+1}$ if and only if $\sigma_i > 2^{n-1}$ (the last bit of $\sigma_i$ equals 1).  
Since $\sigma_{i} = 2\sigma_{i-1} \mod 2^n+1$, the last bit of $\sigma_i$ is the penultimate bit of 
$\sigma_{i-1}$.  By induction, $\sigma_i > \sigma_{i+1}$ if and only if the $(n-i-1)$st bit of $\sigma_1$ equals $1$ for $1\leq i \leq n-1$.  
Thus the structure of $\DD(\sigma)$ determines the binary expansion of $\sigma_1$.
\end{proof}

\subsection{Main Theorem}

For all primes $p$ and $n \in \NN$, 
we find the structure of the Dieudonn\'e module $\DD(X_{p^n})$ of the $p$-torsion group scheme of the Jacobian of the Hermitian curve $X_{p^n}$.
The $\EE$-module structure of $\DD(X_{p^n})$ is determined by its distinct indecomposable factors, which are in bijection with orbits of $\ZZ/(2^n+1) -\{0\}$ under $\Ttwo$, 
and their multiplicities.
The $\EE$-module structure of each indecomposable factor is determined by the combinatorics of the corresponding orbit, as described in Section \ref{Sstructure2}.

\begin{definition}
If $1 \leq t \leq 2^n$ and $s \equiv 2t \bmod 2^n+1$, then ${\rm dim}_k(B_s)={\rm dim}_k(B_t)$ by Theorem \ref{TblockactionFV}(2)(4).
If $\sigma$ is an orbit of $\ZZ/(2^n+1) -\{0\}$ under $\Ttwo$,
its {\it multiplicity} is $m(\sigma):={\rm dim}_k(B_{\sigma_i})$ for any $\sigma_i \in \sigma$.
\end{definition}

The multiplicity $m(\sigma)$ was computed in Lemma \ref{rem:countingblocks}.

\begin{theorem} \label{Torbit}
For all primes $p$ and $n \in \NN$, 
there is a bijection between orbits of $\ZZ/(2^n+1) - \{0\}$ under $\Ttwo$ and distinct indecomposable factors in the Dieudonn\'e module $\DD(X_q)$ of ${\rm Jac}(X_q)[p]$
given by $\sigma \to \DD(\sigma)$.
The multiplicity of $\DD(\sigma)$ in $\DD(X_q)$ is $m(\sigma)$. 
\end{theorem}

\begin{proof}
Suppose $\sigma$ is an orbit of $\ZZ/(2^n+1) - \{0\}$ under $\Ttwo$.
Consider 
\[W_\sigma:={\rm Span}_{\sigma_i \in \sigma} B_{\sigma_i} \subset \Hdr.\]
By Theorem \ref{TblockactionFV}, $W_\sigma$ is stable under the action of $V$ and $F^{-1}$.

Write $\sigma=(\sigma_1, \ldots, \sigma_r)$, choosing $\sigma_1$ to be a local minimum with maximal left distance. 
Let $B=B_{\sigma_1}$.
Define a word $\omega=\omega_r \cdots \omega_1$ in the variables $F^{-1}$ and $V$ as follows: 
$\omega_i = F^{-1}$ if $1 \leq \sigma_i \leq 2^{n-1}$ and $\omega_i = V$ if $2^{n-1}+1 \leq \sigma_i \leq 2^n$.
By Corollary \ref{Cmult2}, the word $\omega$ yields an isomorphism $\omega: B \to B$; (it is $p^{-r}$-linear).
Applying Corollary \ref{Cpermutation} shows that $\omega$ is represented by a {\it generalized permutation matrix}, 
namely a matrix with exactly one non-zero entry in each row and column, with respect to the basis $\BB \cap B$.
This implies that an iterate of $\omega$ can be represented by a diagonal matrix.

In fact, $\omega$ itself can be represented by a diagonal matrix; in other words, that there is a basis of eigenvectors for $\omega$.
To see this, consider the final filtration for the $\EE$-module $W_\sigma$ as described in Section \ref{Seotype}.  
First, $W_\sigma$ has rank $p^{rm}$ where $m={\rm dim}(B)$.
It has a canonical filtration $0=M_0 \subset M_1 \subset \cdots \subset M_r$ where ${\rm dim}(M_i)=im$.
Here each $M_i$ is a union of blocks $B_j$ from the orbit; in particular, $M_r = W_\sigma$ and  $M_1=B$.
The final filtration  $N_1 \subset N_2 \subset \cdots \subset N_{rm}$ is a refinement of the canonical filtration, so $N_{im}=M_i$.  
It is a filtration of $W_\sigma$ as a $k$-vector space which is stable under the action of $V$ and $F^{-1}$ such that $i={\rm dim}(N_i)$.

Let $x_1$ denote a non-zero element of  $N_1 \subset M_1 =B$.  Since the final filtration is stable under $F^{-1}$ and $V$, 
the element $y_1=\omega_1(x_1)=F^{-1}(x_1)$ generates $N_{m+1}/M_1$.  Similarly, $N_{im+1}/M_i$ is generated by an image of 
$x_1$ under a portion of the word $\omega$.  Going through the whole word, $\omega(x_1)$ is a generator for $N_1/N_0$.
Thus $\omega(x_1)$ is a constant multiple of $x_1$.

Thus there is an $\EE$-module isomorphism $W_\sigma \simeq \DD(\sigma)^{m(\sigma)}$.
By Proposition \ref{Pdistinct}, the factors $\DD(\sigma)$ of $\DD(X_q)$ are distinct and are in bijection with orbits $\ZZ/(2^n+1) - \{0\}$ under $\Ttwo$ 
\end{proof}

Recall the definition of {\it break points} from Section \ref{Seotype}.

\begin{corollary}\label{Ckeyvalue}
The Ekedahl-Oort type $\nu$ of $X_q$ has $2^{n-1}$ break points;
in other words, the sequence $\nu_i$ alternates between being constant and increasing on $2^{n-1}$ intervals for $1 \leq i \leq g$.  
This pattern is consistent for all primes $p$, although the formulae for the break points depends on $p$.
\end{corollary}

\begin{proof}
By Theorem \ref{Torbit}, the canonical filtration is constructed by successively adjoining the blocks $B_t$.  The behavior of $F$ and $V$ is consistent across each block.
Thus there are $2^n$ canonical fragments, the first half of which determine break points of $\nu$. 
\end{proof}

\subsection{Indecomposable factors of $\DD(X_{p^n})$ with $a$-number $1$}

For $c \in \NN$, recall from Example \ref{Eorbit1} that $I_{c,1}$ is the unique symmetric ${\rm BT}_1$ group scheme of rank $p^{2c}$ having $p$-rank 0 and $a$-number 1. 
In this section, we find the multiplicity of $\DD(I_{c,1})=\EE/\EE(F^c + V^c)$ in $\DD(X_{p^n})$.
As motivation, note that $\DD(I_{1,1})$ occurs in $\DD(X_{p^n})$ exactly when there is a block $B_t$ such that $F(B_t)=V(B_t)$.  
This can only occur when $n$ is even and $t=(2^{n+1}+2)/3$, in which case the orbit is $\sigma=(t/2, t)$.

We will need the following result about multiplicities of short orbits.
If $W$ is an indecomposable factor of $\DD(X_{p^c})$ and if $n=ck$ for some odd $k \in \NN$, then
$W$ is an indecomposable factor of $\DD(X_{p^n})$ associated with a short orbit by Lemma \ref{Lup}.
The next result compares the multiplicity of $W$ in $\DD(X_{p^c})$ and $\DD(X_{p^n})$.

\begin{proposition} \label{Pmultup}
Suppose $n=ck$ for $k \in \NN$ odd and let $L=(2^n+1)/(2^c+1)$.
The multiplicity $M(\sigma)$ of $\DD(\sigma)$ in $\DD(X_{p^c})$ and
the multiplicity $M(\sigma_L)$ of $\DD(\sigma_L)$ in $\DD(X_{p^n})$
are related by the formula: $M(\sigma_L)=M(\sigma)^k$.
\end{proposition}

\begin{proof}
Note that $M(\sigma)=\dime_k(B_t)$ where $t={\rm min}\{\sigma_i \in \sigma\}$. 
Also, $M(\sigma_L)=\dime_k(B_{Lt})$ because $Lt={\rm min}\{\sigma_i \in \sigma_L\}$. 
Since $t$ is odd, $\vec{b}(t) \in (\ZZ/2)^a$ is the binary expansion of $t-1$.
Note that 
$L= (2^a-1)(2^{n-2a} +2^{n-4a} + \cdots + 2^a) + 1$.  
Now $t(2^a-1) = (t-1)2^a + 2^a-t$ has binary expansion $(\iota(\vec{b}(t)), \vec{b}(t)))$ of length $2a$.  
Thus $Lt-1 = t(2^a-1) (2^{n-2a} +2^{n-4a} + \cdots + 2^a) + (t-1)$  has binary expansion $(\vec{b}(t), \iota(\vec{b}(t)), \vec{b}(t), \ldots, \iota(\vec{b}(t)),\vec{b}(t))$, 
where the sequence has $k$ terms of length $a$.  As $t<2^{n-1}$ the result follows from Lemma \ref{rem:countingblocks}.
\end{proof}

Recall from Proposition \ref{PrankC} that the rank of $\car^i$ on $\Hw$ is $r_{n,i}=p^n(p+1)^i(p^{n-i}-1)/2^{i+1}$.

\begin{corollary} \label{Canumber1}
\begin{enumerate}
\item The Dieudonn\'e module $\DD(I_{n,1})$ occurs with multiplicity $r_{n,n-1}$ in $\DD(X_{p^n})$.

\item The Dieudonn\'e module $\DD(I_{c,1})$ appears as an indecomposable factor of $\DD(X_{p^n})$ if and only if $n=ck$ for some odd $k \in \NN$, 
in which case the multiplicity of $\DD(I_{c,1})$ in $\DD(X_{p^n})$ is $M(I_{c,1}):=(r_{c,c-1})^k$.

\item If $n \in \NN$ is even, then the multiplicity of $\DD(I_{1,1})$ in $\DD(X_{p^n})$ is zero.
If $n \in \NN$ is odd, then the multiplicity of $\DD(I_{1,1})$ in $\DD(X_{p^n})$ is $\left(p(p-1)/2\right)^n$.
\end{enumerate}
\end{corollary}

\begin{remark}
Corollary \ref{Canumber1} is equivalent to the fact that ${\rm Ker}(F^n)={\rm Ker}(V^n)$ has dimension $2g-r_{n,n-1}$ in $H^1_{dR}(X_{p^n})$ or the fact that 
${\rm Im}(F^n)={\rm Im}(V^n)$ has dimension $r_{n,n-1}$ in $H^1_{dR}(X_{p^n})$.
\end{remark}

\begin{proof}
\begin{enumerate}
\item By Example \ref{Eorbit1}, $\DD(I_{n,1})=\DD(\sigma)$ for the orbit $\sigma$ containing $1$. 
Then $M(\sigma)$ equals the dimension of $B_1=V^n B_{2^n}$, which equals the rank $r_{n,n-1}$ of $\car$ on $\Hw$.

\item By part 1, one can suppose that $1 \leq c < n$.  Then ${\rm rank}(\DD(I_{c,1})) < p^{2n}$.
Thus, if $\DD(I_{c,1})$ occurs in $\DD(X_{p^n})$, 
then $\DD(I_{c,1})=\DD(\hat{\sigma})$ for a short orbit $\hat{\sigma}$ of 
$\ZZ/(2^n+1) -\{0\}$.  By Lemma \ref{Ldown}, $n=ck$ for some $k \in \NN$ odd.
Suppose $n=ck$ for some $k \in \NN$ odd.  By part 1, $\DD(I_{c,1})$ appears in 
$\DD(X_{p^c})$ with multiplicity $r_{c,c-1}$.  The result then follows from Lemma \ref{Lup} and Proposition \ref{Pmultup}.

\item This follows from part 2, setting $c=1$.
\end{enumerate}
\end{proof}

As an example, consider the case $n=4$, which involves the rank 8 group scheme $I_{4,3}$ from Example \ref{EI43}.

\begin{example} \label{Ecasen=4}
The Dieudonn\'e module $\DD(X_{p^4})$ of ${\rm Jac}(X_{p^4})[p]$ is: 
\begin{equation}
\DD(X_{p^4}) = (\EE/\EE(F^4+V^4))^{r_{4,3}} \oplus (\DD(I_{4,3}))^{r_{4,1}-3r_{4,3}}.
\end{equation}
\end{example}

\begin{proof}
The orbit $\sigma=\{1,2,4,8,16,15,13,9\}$ has $\DD(\sigma)=\DD(I_{4,1})$.
The multiplicity of $\DD(I_{4,1})$ is determined by Corollary \ref{Canumber1}(1).
There is one other orbit $\sigma'=\{3,6,12,7,14,11,5,10\}$ of $\Ttwo$ on $\ZZ/17$.
By Example \ref{EI43}, $\DD(\sigma')=\DD(I_{4,3})$.
The multiplicity of $\DD(I_{4,3})$ equals $(2g-8r_{4,3})/8$. 
\end{proof}

\section{Applications} \label{Sapplication}

\subsection{Decomposition of Jacobians of Hermitian curves} \label{SJacDec}

The fact that ${\rm Jac}(X_{p^n})$ is supersingular is equivalent to the fact that it decomposes, up to isogeny, into a product of supersingular elliptic curves:
\[{\rm Jac}(X_{p^n}) \sim \times_{i=1}^g E_i.\] 
A more refined problem is about the decomposition of ${\rm Jac}(X_{p^n})$ up to isomorphism.  Consider an isomorphism
\[{\rm Jac}(X_{p^n}) \simeq \times_{i=1}^N A_i\]
of abelian varieties without polarization, where each $A_i$ is indecomposable and $g=\sum_{i=1}^N \dime(A_i)$.

When $n=1$, Section \ref{Ecasen=1} and \cite[Theorem 2]{oort75} imply that the Jacobian of $X_p$ is isomorphic to a product of supersingular elliptic curves: 
\[{\rm Jac}(X_p) \simeq \times_{i=1}^g E_i.\]
For $n \geq 2$, we did not find any results about the decomposition of ${\rm Jac}(X_{p^n})$ up to isomorphism in the literature.
In this section, we use Theorem \ref{Torbit} to provide constraints on this decomposition.

\subsubsection{Elliptic rank}

If $A$ is an abelian variety, its {\it elliptic rank} is the largest non-negative integer $r$ such that there exist elliptic curves 
$E_1, \ldots, E_r$ and an abelian variety $B$ of dimension $g-r$ and an isomorphism $A \simeq B \times (\times_{i=1}^r E_i)$ of abelian varieties without polarization.

\begin{application} \label{Aellipticrank}
If $n$ is even, then the elliptic rank of ${\rm Jac}(X_{p^n})$ is $0$.
If $n$ is odd, then the elliptic rank of ${\rm Jac}(X_{p^n})$ is at most $\left(p(p-1)/2\right)^n$.
\end{application}

In fact, the elliptic rank of ${\rm Jac}(X_{p^n})$ is exactly $\left(p(p-1)/2\right)^n$ when $n$ is odd.  The proof of this will be left for a later paper.

\begin{proof}
If ${\rm Jac}(X_{p^n}) \simeq B \times (\times_{i=1}^r E_i)$, then each $E_i$ is supersingular and $\DD(E_i) \simeq \EE/\EE(F+V)$.
The result follows from Corollary \ref{Canumber1}(3) since the elliptic rank is bounded by the multiplicity of $\DD(I_{1,1})$ in $\DD(X_{p^n})$.
\end{proof}

\subsubsection{A partition condition on the decomposition}

We determine a partition condition on the decomposition of the Jacobian ${\rm Jac}(X_{p^n})$ up to isomorphism, starting with a simple-to-state application.

\begin{application} \label{Apower2}
Suppose $n=2^e$ for some $e \in \NN$ and suppose ${\rm Jac}(X_{p^n}) \simeq \times_{i=1}^N A_i$.
Then $n \mid {\rm dim}(A_i)$ for $1 \leq i \leq N$ and $N \leq g/n$.
In particular, when $n=2$, then ${\rm dim}(A_i)$ is even for all $1 \leq i \leq N$.
\end{application}

\begin{proof}
If $n=2^e$, then all orbits $\sigma$ of $\ZZ/(2^n +1) -\{0\}$ have length exactly $2n$.
By Lemma \ref{LDsigma}, $\dime(\DD(\sigma))=2n$.
Also $\DD(A_i)$ has dimension $2\dime(A_i)$ and is a direct sum of
Dieudonn\'e modules of dimension $2n$.
\end{proof}

\begin{definition} \label{Dpartition}
Consider two partitions $\eta_J$ and $\eta_{\DD}$ defined as follows.
If $J \simeq \times_{i=1}^N A_i$,
where each $A_i$ is an indecomposable abelian variety, let
$\eta_J = \{{\rm dim}(A_i) \mid 1 \leq i \leq N\}$.
If $\DD(X_{p^n})= \oplus_{i=1}^{\delta} D_i$, where each $D_i$ is an indecomposable symmetric 
Dieudonn\'e module, let 
$\eta_{\DD} = \{{\rm dim}(D_i) \mid 1 \leq i \leq \delta\}$. 
\end{definition}

It is clear that the partition $\eta_{\DD}$ is a refinement of the partition $\eta_J$.
For any $q$, this observation can be used to compute a lower bound for the partition $\eta_J$
which is the set of dimensions of the indecomposable factors in the decomposition of ${\rm Jac}(X_{p^n})$ up to isomorphism.  
In particular, this yields the upper bound $N \leq \sum_{\sigma} m(\sigma)$.
For example, when $n=3$, then $N \leq g-2r_{3,2} \sim g/2$.












\subsection{Application to Selmer groups} \label{SSelmer}

Let $A$ be an abelian variety defined over the function field $K$ of $X_q$ with $q=p^n$.  
Let $f : A \to A'$ be an isogeny of abelian varieties over $K$.
Recall that the Tate-Shafarevich group $\Sha(K,A)$ is the kernel of $H^1(K, A) \to \prod_{v} H^1(K_v, A)$ where the product is taken over all places $v$ of $K$.
Let $\Sha(K, A)_f$ be the kernel of the induced map $\Sha(K, A) \to \Sha(K, A')$.
Also define the local Selmer group ${\rm Sel}(K_v, f)$ to be the image of the coboundary map $A'(K_v) \to H^1(K_v, \Ker(f))$ 
and the global Selmer group to be the subset of $H^1(K, \Ker(f))$ which restrict to elements of ${\rm Sel}(K_v, f)$ for all $v$.   
There is an exact sequence 
\[0 \to A'(K)/f(A(K)) \to {\rm Sel}(K,f) \to \Sha(K, A)_f \to 0.\]

In \cite[Theorems 1 \& 2]{dum99}, the author determines the group structure of $\Sha$ in the case when $A$ is ${\rm Jac}(X_q)$ or $A$ is a
supersingular elliptic factor of ${\rm Jac}(X_q)$.  Here is a quick application about this topic.

\begin{application} \label{ASelmer}
Let $E$ be a constant elliptic curve over the function field $K$ of $X_q$.  
\begin{enumerate}
\item If $E$ is ordinary, then ${\rm Sel}(K, [p])$ has rank $2r_{n,1}=p^n(p+1)(p^{n-1}-1)/2$.
\item If $E$ is supersingular, then ${\rm Sel}(K, [p])$ has rank $0$ if $n$ is even and rank $\left(p(p-1)/2\right)^n$ if $n$ is odd.
\end{enumerate}
\end{application}

\begin{proof}
\begin{enumerate}
\item
The result follows from Proposition \ref{PrankC} because the rank of ${\rm Sel}(K, [p])$ is twice the rank of $\car$ \cite[Proposition 3.3]{Ulmer}.  

\item
The result follows from Corollary \ref{Canumber1}(3) because
 the rank of ${\rm Sel}(K, [p])$ is the dimension of $\Ker(F+V)$ on $\Hdr$ \cite[Proposition 4.3]{Ulmer}.  
\end{enumerate}
\end{proof}

\subsection{Application about the supersingular locus} \label{Ssupersingular}

The moduli space ${\mathcal A}_g$ of principally polarized abelian varieties of 
dimension $g$ can be stratified by Ekedahl-Oort type into locally closed strata.
By \cite[Lemma 10.13]{Ohandbook}, the stratum for the Ekedahl-Oort type $\nu$ is contained in the supersingular locus $S_g$ if and only if $\nu_{s}=0$ where $s=\lceil g/2 \rceil$. 

Each generic point of $S_g$ has $a$-number $1$ \cite[Section 4.9]{LO}. 
By Example \ref{Eorbit1}, the unique Ekedahl-Oort type with $p$-rank $0$ and $a$-number $1$ has $\nu_s =s-1$ which is not zero for $g \geq 3$.
Thus this Ekedahl-Oort stratum intersects but is not contained in $S_g$.  

For all $p$, we give infinitely many new examples of Ekedahl-Oort strata which intersect but are 
not contained in $S_g$.  What is significant is that each has large $a$-number, 
namely just a bit smaller than $g/2$.
Note that $a \leq \lfloor (g-1)/2 \rfloor$ is the smallest upper bound for $a$ which guarantees 
that $\nu_s \not = 0$.  
 
\begin{application}
Let $q=p^n$ with $n \geq 3$ and let $g=q(q-1)/2$.
The Hermitian curve $X_q$ has $a$-number $\frac{g}{2}[1-\frac{p}{q}\frac{p^{n-2}-1}{q-1}]$.
Its Ekedahl-Oort stratum intersects, but is not contained in, the 
supersingular locus of $\mathcal{A}_g$.
\end{application}

\begin{proof}
The Jacobian of the Hermitian curve $X_{p^n}$ is supersingular and has dimension $g$.
Let $\nu$ be its Ekedahl-Oort type 
and let $\eta$ be the strata of ${\mathcal A}_g$ with Ekedahl-Oort type $\nu$.
By Proposition \ref{PrankC}, $\nu_i=0$ if and only if $i \leq r_{n, n-1}=p^n(p+1)^{n-1}(p-1)/2^{n}$.  
By \cite[Lemma 10.13]{Ohandbook}, $\eta \subset S_g$ if and only if $\nu_{s}=0$ where
$s=\lceil g/2 \rceil$.  This condition is not satisfied for $n \geq 3$.  
\end{proof}

\bibliographystyle{amsalpha}
\bibliography{Hermitian}

Rachel Pries, Colorado State University, Fort Collins, CO 80521, pries@math.colostate.edu

Colin Weir, Simon Fraser University, Vancouver, BC, Canada,  V5A 1S6, colin\_weir@sfu.ca

\end{document}